\documentclass{siamart1116}

\usepackage{graphicx}
\usepackage{latexsym,amsfonts,amssymb,amsmath,theorem,subeqnarray}
\usepackage{framed} 
\usepackage{graphicx,psfrag,epsf} 
\usepackage{subfigure}
\usepackage{enumitem}
\usepackage{attachfile}

\DeclareMathOperator*{\vol}{Vol}

\DeclareMathOperator*{\co}{co}
\DeclareMathOperator*{\cl}{cl}

\DeclareMathOperator{\inte}{int}
\DeclareMathOperator{\proj}{\mathcal{P}}
\DeclareMathOperator{\ri}{\text{ri}}
\newcommand{\set}{\mathcal}

\newtheorem{assumption}{Assumption}
\newsiamremark{remark}{Remark}
\usepackage{lipsum}
\usepackage{amsfonts}
\usepackage{graphicx}
\usepackage{epstopdf}
\usepackage{algorithmic}
\ifpdf
  \DeclareGraphicsExtensions{.eps,.pdf,.png,.jpg}
\else
  \DeclareGraphicsExtensions{.eps}
\fi

\numberwithin{theorem}{section}

\newcommand{\TheTitle}{A Fast Gradient and Function Sampling Method for Finite Max-Functions} 
\newcommand{\TheAuthors}{Elias Salom\~{a}o Helou, Sandra S. Santos, and Lucas E. A. Sim\~{o}es}

\headers{A Fast Gradient and Function Sampling Method}{\TheAuthors}

\title{{\TheTitle}\thanks{Submitted to the editors DATE.
\funding{This work was supported by Brazilian Funding Agencies \emph{Funda\c{c}\~ao de Amparo \`a Pesquisa do Estado de S\~ao Paulo} - FAPESP (grants 2013/07375-0, 2013/05475-7, 2013/07375-0, 2013/14615-7, 2016/22989-2 and 2016/24286-9), \emph{Conselho Nacional de Desenvolvimento Cient\'{i}fico e Tecnol\'{o}gico} - CNPq (grants 311476/2014-7 and 302915/2016-8) and PRONEX Optimization.}}}

\author{
  Elias Salom\~{a}o Helou\thanks{Institute of Mathematical Sciences and Computation,
University of S\~{a}o Paulo. S\~{a}o Carlos - SP, Brazil.
    (\email{elias@icmc.usp.br}).}
  \and
  Sandra A. Santos\thanks{Department of Applied Mathematics, University of Campinas. 
              Campinas - SP, Brazil.\quad\quad\quad\quad(\email{sandra@ime.unicamp.br},
    \email{simoes.lea@gmail.com}).}
  \and
  Lucas E. A. Sim\~{o}es\footnotemark[3]
}

\usepackage{amsopn}

\ifpdf
\hypersetup{
  pdftitle={\TheTitle},
  pdfauthor={\TheAuthors}
}
\fi

\begin{document}

\maketitle

% REQUIRED
\begin{abstract}
  This paper tackles the unconstrained minimization of a class of nonsmooth and nonconvex functions that can be written as finite max-functions. A gradient and function-based sampling method is proposed which, under special circumstances, either moves superlinearly to a minimizer of the problem of interest or superlinearly improves the optimality certificate. Global and local convergence analysis are presented, as well as illustrative examples that corroborate and elucidate the obtained theoretical results.
\end{abstract}

% REQUIRED
\begin{keywords}
  nonsmooth nonconvex optimization, gradient sampling, local superlinear convergence, global convergence, unconstrained minimization
\end{keywords}

% REQUIRED
\begin{AMS}
  65K10, 90C26
\end{AMS}

\section{Introduction}  

Problems involving continuous nonsmooth functions arise in many fields of science~\cite{MAY09,PJS10,WAC09}, playing a primary or a secondary role (e.g. subproblems) in different areas. A wide class of problems needs to cope with one or more minimizations of convex nonsmooth functions~\cite{MOP14,OKZ13}, which has been successfully solved by well established optimization algorithms known as \textit{Bundle Methods}~\cite{BAW75,KWL85,MAK02}. However, a significant amount of problems involve minimizations of nonsmooth functions that are also nonconvex~\cite{DOA12,DSD09}, a property that usually introduces an undesirable complexity to the implementation of the aforementioned method. Nevertheless, we can also find algorithms based on bundle ideas~\cite{GGM09,KWL96} for such functions.

Recently, an algorithm known as \textit{Gradient Sampling} (GS)~\cite{BLO05,KWL07} has gained attention for providing good alternatives to the difficulties that the Bundle Methods need to deal with if the function is not convex~(see \cite{MAK02,OLS14} and references therein). Basically, the functioning of GS is very close to the steepest descent method for smooth functions, since it works in every iteration with a descent direction computed just with first order information and it finds the next iterate by a line search procedure (in fact, when a nonnormalized version of GS is used to solve a smooth optimization problem, its step asymptotically recovers the direction taken by the steepest descent method). In contrast to the Bundle Method, the GS does not work with a memory of the past iterations, but it tries to gain information about the function by computing gradients at some sampled points obtained in each iteration. This behavior is less complex than keeping a history of the last iterations, since in the nonconvex case, it is hard to determine whether a past iteration is contributing to construct a good model of the objective function or it is so far from the current iteration that its incorporation to the model might lead to an erroneous information. As a counterpart, by evaluating the gradients at the sampled points, the GS has a significant cost per iteration.

Since we can interpret the GS algorithm as a generalization of the steepest descent method, it is reasonable to think that, in the best-case scenario, the method would have linear local convergence~\cite{HSS17}. Therefore, this leads to a natural question: would it be possible to have a GS algorithm that can be understood as a generalization of Newton's (or quasi-Newton) method for nonsmooth functions, meaning that it would locally converge faster than linearly?

This manuscript has the intent to start answering this question. As we shall see, the answer is, at least, partially affirmative. In fact, there are recent studies that have introduced GS-like algorithms with quasi-Newton techniques~\cite{CUO12,CUQ13,CUQ15}, however there are no proofs nor numerical results that corroborate a rapid local convergence. Therefore, our affirmative answer is directly linked to the property that, in a good sampling condition and, for a special class of nonsmooth functions, the method will move superlinearly in some sense.

One might view our method as a GS algorithm that incorporates some elements of Bundle Methods developed over the years~\cite{GRM98,LUV98}, but still keeps the GS facilities to handle nonconvex functions. This last characteristic is in agreement with Kiwiel's expectation~\cite{KWL07}\\
\begin{quote}
``We believe, however, that deeper understanding of their [GS and Bundle Methods] similarities and differences
should lead to new variants."\\
\end{quote}

In order to prove a rapid local convergence result, the theory developed in this manuscript is based on the $\set V\set U$-decomposition of the space~\cite{LEW02,MIS99}. However, the method does not need to compute an estimate of such spaces. Roughly speaking, we show that our trust-region algorithm emulates the quasi-Newton techniques into the $\set U$-space (a subspace where the objective function is locally smooth), whereas it combines effective cutting-plane features~\cite{FGG04,KEL60} into the $\set V$-space (the orthogonal complement of the $\set U$-space). For this purpose, we need not only to evaluate the gradients at the sample points, but also their respective function values. This procedure does not produce a significant increase in computational time, since, in most cases, the computational effort of evaluating the function value is fundamental in evaluating the gradient as well, so, by computing the gradient, one can obtain the function value essentially for free.

As a consequence of our attempt to move superlinearly to the solution of the optimization problem, the iterations of the proposed algorithm are more expensive when compared to the GS method. Therefore, although the global convergence of our algorithm is proven, the method should be viewed as an accelerator of the local convergence speed of the GS algorithm. Consequently, we propose that a potential user should use the GS method in the first iterations and switch to our algorithm in the final iterations. Fortunately, by the way the GS was designed, this transition moment can be well determined.

Finally, we believe that the results obtained in this text are a step further into the study of a practical algorithm with rapid local convergence to minimize nonsmooth and nonconvex functions (important studies on the matter for nonsmooth and convex functions can be found in~\cite{LEM82,LOS00,LES97,MIS05}). The pursuit for such an algorithm has raised many researchers' efforts (an enlightening review can be found in~\cite{MIS12}) and up to our knowledge there is no method in the literature that fulfills those features. A future work assessing its performance in an extensive class of nonsmooth functions is needed to determine how efficient the proposed algorithm is. For now, we limit ourselves to the global and local convergence theory and the presentation of some illustrative examples.   

For clarity, before we start to expose the main ideas of this study, we present some notations that appear along this manuscript:
\begin{itemize}
\item{$\co \set{X}$ is the convex hull of $\set{X}$;}
\item{$\cl \set{X}$ is the closure of $\set{X}$;}
\item{$\inte \set{X}$ is the interior of $\set{X}$;}
\item{$|\set{X}|$ is the cardinality of $\set{X}$;}
\item{$\set{B}(x,r)$ is the Euclidean closed ball with center at $x$ and radius $r$;}
\item{$\|\cdot\|$ is the Euclidean norm in $\mathbb{R}^n$;}
\item{$\|x\|_{H} := \sqrt{x^THx}$, for any symmetric positive definite matrix $H$;}
\item{$e$ is a vector with ones in all entries;}
\item{$\mathcal{P}[x\in \set X]$ is the probability of $x$ to be in $\set X$, whereas $\mathcal{P}[x\in \set X~|~x\in \set Y]$ is the conditional probability of $x$ to be in $\set X$ given that $x\in \set Y$.} 
\end{itemize}

\section{Basic concepts and the GS algorithm}\label{sec:basic}

The GS method has the goal of solving the following unconstrained optimization problem
\begin{equation}\label{eq:main-prob}
\min_{x\in \mathbb{R}^n} f(x)\text,
\end{equation}
where $f:\mathbb{R}^n\rightarrow \mathbb{R}$ is a locally Lipschitz function, continuously differentiable in an open dense subset with full measure $\set D\subset \mathbb{R}^n$. The function $f$ is not necessarily convex.

For a map with the properties above, it is possible to define the Clarke's subdifferential set for $f$ at $x$~\cite{CRK90,CRK98}. This set can be interpreted as a generalization of the gradient for nonsmooth functions.
\begin{definition}[Subdifferential set, subgradient, stationary point]
The set given by
\begin{equation*}
\overline{\partial}f(x) := \co\left\{\lim_{j\rightarrow \infty} \nabla f(x_j)~\displaystyle |~ x_j\rightarrow x, x_j\in \set D\right\}
\end{equation*}
is called the Clarke's subdifferential set of $f$ at $x$ and any $v\in \overline{\partial}f(x)$ is known as a subgradient of $f$ at $x$. Moreover, if $0\in \overline{\partial}f(x)$, then we say that $x$ is a stationary point for $f$.
\end{definition}
A set that fits best with the idea of sampled points and is more general than the previous one can be defined~\cite{GOL77}.
\begin{definition}[$\epsilon$-Subdifferential set, $\epsilon$-subgradient, $\epsilon$-stationary point] The $\epsilon$-subdifferential set of~$f$~at $x$ is given by
\begin{equation*}
\overline{\partial}_\epsilon f(x) := \text{co~} \overline{\partial} f(\set B(x,\epsilon))\text.
\end{equation*}
Any $v\in \overline{\partial}_\epsilon f(x)$ is known as an $\epsilon$-subgradient of $f$ at $x$. Moreover, if $0\in \overline{\partial}_\epsilon f(x)$, then we say that $x$ is an $\epsilon$-stationary point for $f$.
\end{definition}
With a great importance for our study, we present the generalized directional derivative for the function $f$~\cite{CRK90}.
\begin{definition}[Generalized directional derivative]\label{def:gdd}
The generalized directional derivative of a continuous locally Lipschitz function $f:\mathbb{R}^n\rightarrow \mathbb{R}$ at $x$ in the direction $v\in \mathbb{R}^n$ is given by 
\begin{equation*}
f^\circ(x;v):= \underset{t\downarrow 0}{\limsup_{y\rightarrow x}}~\frac{f(y+tv)-f(y)}{t}\text.
\end{equation*} 
\end{definition}

Finally, it is possible to link Definition~\ref{def:gdd} with the subdifferential set. Indeed, the following relation holds~\cite{CRK90}
\begin{equation*}
f^\circ(x;v) = \max\{s^Tv~|~s\in \overline{\partial} f(x)\}\text.
\end{equation*}
With the above sets in mind, one can interpret the sampled points used in GS method as an attempt to approximate the $\epsilon$-subdifferential set of $f$ at $x$~\cite[Theorem 2.1] {BLO02a}. 

For a more complete idea of the GS functioning, we present the nonnormalized version of the GS algorithm~\cite{KWL07}.

\begin{algorithm}
\caption{Nonnormalized version of the GS method.}
\begin{itemize}[leftmargin=1.4cm]
\item[\textbf{Step 0.}]{Given $x_0\in \set D$, $m \in \mathbb{N}$ with $m\geq n+1$, fixed real numbers $0\leq \nu_{\text{opt}} < \nu_0$, $0\leq \epsilon_{\text{opt}} < \epsilon_0$ and $0<\theta_\nu,\theta_\epsilon,\gamma,\beta<1$, set $k=0$.}
\item[\textbf{Step 1.}]{Choose $\left\{x_{k,1},\hdots,x_{k,m}\right\}\subset \set B(x_{k},\epsilon_k)$ with randomly, independently and\\ uniformly sampled elements. If $\left\{x_{k,1},\hdots,x_{k,m}\right\} \not\subset \set D$, then STOP!}
\item[\textbf{Step 2.}]{Set $G_k = [\nabla f(x_{k}),~\nabla f(x_{k,1}),\hdots,~\nabla f(x_{k,m})]$ and find $g_k = G_k\lambda^k$, where $\lambda^k$\\ solves 
\begin{equation*}
\begin{split}
\min_\lambda  ~~& \frac{1}{2} \lambda^TG_k^TG_k\lambda \\
\text{s.t.}          ~~& e^T\lambda = 1\text,~ \lambda \geq 0\text.
\end{split}
\end{equation*} 
\\}
\item[\textbf{Step 3.}]{If $\|g_k\| \leq \nu_{\text{opt}}$ and $\epsilon_k \leq \epsilon_{\text{opt}}$, then terminate. Otherwise, if $\|g_k\| \leq \nu_k$, then \\ $\epsilon_{k+1} = \theta_\epsilon\epsilon_k$, $\nu_{k+1} = \theta_\nu\nu_k$, $t_k = 0$, $x_{k+1} = x_k$ and go to Step 6. }
\item[\textbf{Step 4.}]{Do a backtracking line search and find the maximum $t_k\in \{1,\gamma,\gamma^2,\hdots\}$ \\ such that
$$ f( x_k + t_kd_k ) < f( x_k ) - \beta t_k \|g_k\|^2 \text{,~~where~}d_k = -g_k\text,$$
and set $\epsilon_{k+1} = \epsilon_k$ and $\nu_{k+1} = \nu_k$.\\}
\item[\textbf{Step 5.}]{If $x_k + t_kd_k \in \set D$, then set $x_{k+1} = x_k + t_kd_k$. Otherwise, find 
\begin{equation*}
x_{k+1}\in \set B(x_k~+~t_kd_k,\min\{t_k,\epsilon_k\}\|d_k\|)\cap \set D\text, 
\end{equation*}
such that $f( x_{k+1} ) < f( x_k ) - \beta t_k\|g_k\|^2 $.}
\item[\textbf{Step 6.}]{Set $k\leftarrow k+1$ and go back to Step 1.}
\end{itemize}
%\SetAlCapFnt{\small\textit}
\label{model-alg}
\end{algorithm}

Since the sampled points are chosen in an uniform and independent way, one can show that the GS method, with probability one, will never stop due to Step 1. Moreover, it is possible to show that if $x_k\in \set D$, then the vector $d_k$ used at Step 4 is a descent direction for $f$ at $x_k$~\cite{BLO05}, which evinces the importance of Step~5 for the finiteness of the line search procedure (in fact, this procedure is a delicate matter~\cite{HSS16}). Moreover, given the random nature of the method, nondeterministic results of convergence with probability one are expected~\cite{KWL07}.

Once we have presented some basic notions about nonsmooth functions and the GS methods, we are able to proceed with the main ideas of this paper. 

\section{Motivation and the new algorithm}
%In this section, with the help of a motivational example, we discuss some of the ideas presented in our method. We analyze why the GS method not always takes full advantage of the sampled points, and lastly, we introduce the new algorithm.

Henceforward, we will be interested in solving a class of problems more structured than~\eqref{eq:main-prob}. Let us consider the minimax optimization problem 
\begin{equation}\label{eq:minimax-prob}
\min_{x\in \mathbb{R}^n} \left(f(x):=\max_{1\leq i \leq p}\{\phi_i(x)\}\right)\text,
\end{equation}
where the functions $\phi_i:\mathbb{R}^n\rightarrow \mathbb{R}$ are all of class $C^2$, but they are not necessarily known. Here, we only ask that the function $f$ may be represented as a maximum of functions, i.e., the functions $\phi_i$ are not inputs for the method. This situation is distinct from the case in which the functions that comprise $f$ are known. For such a case, many studies have been developed (see~\cite{DUP13} and references therein).

\subsection{Motivational example} Suppose we have $f(x) = |x| = \max\{x,-x\}$ and we want to start an iteration of Algorithm~\ref{model-alg}. If
\begin{equation*}
m = 2 ,~~\epsilon_0 = 1,~~\epsilon_{\text{opt}}<1,~~x_0 = 0.5,~~x_{0,1} < 0\text{~~and~~}x_{0,2} > 0\text,
\end{equation*}
then $f'(x_{0,1}) = -1$, $f'(x_{0,2}) = 1$ and $g_0 = 0$ in Step 2. Consequently, by Step~3, we skip Steps~4 and 5 and go directly to Step 6, which starts a new iteration. Although this routine indicates that we have an $\epsilon_0$-stationary point for $f$, this procedure does not allow us to move. Moreover, it prevents the algorithm to take an action when it has a complete information about the function, that is, when we have points sampled in the sets 
\begin{equation*}
X^- = \{x\in \mathbb{R}~|~x < 0\}\text{~~and~~}X^+ = \{x\in \mathbb{R}~|~x > 0\}\text. 
\end{equation*}
As a consequence, we see that the method only gets a chance to move when either $x_k$ and the sampled points are all in $X^-$ or all in $X^+$. Moreover, in this scenario, the GS method behaves exactly as the steepest descent method.

This undesirable behavior can be explained by the lack of information about the function values at the sampled points. Indeed, taking a careful look into the quadratic optimization problem that is solved in Step 2, it is possible to see that its dual problem is given by
\begin{equation*}
\begin{split}
\min_{d,z}  ~~& z + \frac{1}{2} d^Td \\
\text{s.t.}          ~~& G_k^Td \leq ze\text,
\end{split}
\end{equation*} 
where $z\in \mathbb{R}$ and $d\in \mathbb{R}^n$. Equivalently, considering $x_{k,0}:= x_k$, the same direction $d_k$ can be obtained if we solve  
\begin{equation}\label{eq:qp-primal}
\min_{d \in \mathbb{R}^n} \max_{0\leq i \leq m} \left\{f(x_{k}) + \nabla f(x_{k,i})^Td + \frac{1}{2} d^Td\right\}\text.
\end{equation} 
Notice, however, that if we use the function values of each sampled point instead of $f(x_k)$ and an enriched second-order information (other GS-like methods use different second-order approaches~\cite{CUO12,CUQ13}), i.e., if we solve
\begin{equation}\label{eq:qp-max-max}
\min_{d \in \mathbb{R}^n} \max_{1\leq i \leq m} \left\{f(x_{k,i}) + \nabla f(x_{k,i})^T(x_k + d - x_{k,i}) + \frac{1}{2} d^TH_kd\right\}\text,
\end{equation}
we would have a better model for the function $f$ than the original one (closer to a cutting-plane method). Furthermore, the new quadratic optimization problem allows us to move when we have sampled in both ``faces" of $f$, that is, in $X^-$ and $X^+$. Lastly, observe that in~\eqref{eq:qp-max-max}, we do not use the objective function value at the current iterate $x_k$ neither the gradient $\nabla f(x_k)$. As we shall see later, these omissions do not prevent the algorithm to converge and introduce an advantage over the GS method, since the differentiability check inside Step 5 is no longer necessary.

Unfortunately, this new quadratic programming problem comes at a price: the vector $d_k$ might not be a descent direction for $f$ at $x_k$ (especially under a bad sampling condition), a property that is always true if we solve~\eqref{eq:qp-primal}. Therefore, to have an algorithm that uses the function values at all sampled points, we must overcome this issue.

\subsection{New algorithm} In order to surpass the difficulty of not having a descent direction under a bad sampling, we replace the Armijo's line search by a trust-region procedure. Besides, aiming at a smooth problem, instead of dealing with~\eqref{eq:qp-max-max}, we solve at each iteration the following quadratic optimization problem
\begin{equation}\label{eq:qp-max}
\begin{split}
\min_{d,z}  ~~& z + \frac{1}{2} d^TH_kd \\
\text{s.t.}          ~~& \tilde{f}_k + G_k^Td \leq ze\\
              ~~& \|d\|_\infty \leq \Delta_k\text,
\end{split}
\end{equation} 
where $\tilde{f}_k = [f(x_{k,1}) + \nabla f(x_{k,1})^T(x_k - x_{k,1}),\hdots,f(x_{k,m})+ \nabla f(x_{k,m})^T(x_k - x_{k,m})]^T$, $G_k = [\nabla f(x_{k,1})~\hdots \nabla f(x_{k,m})]$ and $\|d\|_\infty \leq \Delta_k$ stands for the trust-region constraints, for some $\Delta_k > 0$. Consequently, its dual optimization problem, after a changing of variables, can be viewed as 
\begin{equation*}\label{qp-max-dual}
\begin{split}
\max_{\lambda,\omega}~~ & \lambda^T\tilde{f}_k - \frac{1}{2}(G_k\lambda + \omega)^TH_k^{-1}(G_k\lambda + \omega) - \Delta_k\|\omega\|_1\\
\text{s.t.~~} & \lambda^Te = 1 \\
             & \lambda \geq 0\text,
\end{split}
\end{equation*}
where $\lambda\in \mathbb{R}^m$ and $\omega\in \mathbb{R}^n$ are the dual variables. With these modifications in mind, we introduce the proposed algorithm (Algorithm~\ref{vugs-alg}), also referred as GraFuS, which stands for Gradient and Function Sampling. Together with the exhibition of our new method, we must highlight that the generated sequence of function values might not be monotone decreasing (the reason for this choice will be better explained in the local convergence subsection). Additionally, with the same argument used for the GS method, GraFuS, with probability one, will never stop at Step 1.

\begin{algorithm}[h!]
\caption{Gradient and Function Sampling-based method (GraFuS).}
\begin{itemize}[leftmargin=1.4cm]
\item[\textbf{Step 0.}]{Set $k,l = 0$, $x_0\in \mathbb{R}^n$, $1\leq\sigma_0\leq 2$, $m \in \mathbb{N}$ with $m\geq n+1$ and fixed real \\ numbers $\gamma_\epsilon > 0, \gamma_\Delta  > 0$, $0 < \nu_0,\theta,\rho,\delta < 1$, $0 \leq \nu_\text{opt} < \nu_0$ and $\varrho > 1$. Define \\ the initial sampling radius as $\epsilon_{0,0} = \gamma_\epsilon\nu_0$ and the maximum step size as \\ $\Delta_{0,0} = \gamma_\Delta\nu_0$.}
\item[\textbf{Step 1.}]{Choose $$\left\{x_{k,1}^l,\hdots,x_{k,m}^l\right\}\subset \set B\left(x_{k},(\epsilon_{k,l})^{\sigma_k}\right)$$ with randomly, independently and uniformly sampled elements. \\ If $\left\{x_{k,1}^l,\hdots,x_{k,m}^l\right\} \not\subset \set D$, then STOP! Otherwise, set $\tilde{f}_{k,l}\in \mathbb{R}^m$ with
\begin{equation*}
\left(\tilde{f}_{k,l}\right)_j = f(x_{k,j}^l) + \nabla f(x_{k,j}^l)^T(x_k - x_{k,j}^l)\text{,}~1\leq j\leq m\text,  
\end{equation*}
and
\begin{equation*}
G_{k,l} = [\nabla f(x_{k,1}^l),\hdots,\nabla f(x_{k,m}^l)]\text.
\end{equation*}
}
\item[\textbf{Step 2.}]{Find $(d_{k,l},z_{k,l})$ and $(\lambda_{k,l},\omega_{k,l})$ that solve, respectively, \eqref{eq:qp-max} and its dual\\ problem, where $H_k\in \mathbb{R}^{n\times n}$ is a symmetric and positive definite matrix.\\}
\item[\textbf{Step 3.}]{If $\|H_k^{-1}G_{k,l}\lambda_{k,l}\| \geq \nu_k$ and $\Delta_{k,l} < +\infty$, then proceed to the next step. \\Otherwise,
\begin{equation*}
\begin{array}{ll}
\bullet \text{ if } \|d_{k,l}\|_\infty < \Delta_{k,l}: & \text{choose } \sigma_{k+1}\in [1,2] \text{, set the optimality certificate}\\
&\nu_{k+1} = \min\{\max\{\|H_k^{-1}G_{k,l}\lambda_{k,l}\|,(\nu_k)^\varrho\},\delta\nu_k\} \text{ and }\\
&\text{go to Step 6.}\\
\bullet \text{ if } \|d_{k,l}\|_\infty = \Delta_{k,l}: & \text{set } \Delta_{k,l+1} = +\infty\text,~G_{k,l+1} = G_{k,l}\text,~ \tilde{f}_{k,l+1} = \tilde{f}_{k,l}\text, \\
&l\leftarrow l+1\text{ and go back to Step 2.}
\end{array}
\end{equation*}
}
\item[\textbf{Step 4.}]{
Compute
\begin{equation*}
\text{Ared}_{k,l} := f(x_k) - f(x_k+d_{k,l})
\end{equation*}
and\vspace{0.3cm}\\
$\text{Pred}_{k,l} := \displaystyle\max_i \left\{f(x_{k,i}^l)+\nabla f(x_{k,i}^l)^T(x_k - x_{k,i}^l)\right\} - \left(z_{k,l} + \frac{1}{2}d_{k,l}^TH_kd_{k,l}\right)\text.$
\vspace{0.3cm}\\}
\item[\textbf{Step 5.}]{If $\text{Ared}_{k,l} \leq \rho\text{Pred}_{k,l}$, then set $\Delta_{k,l+1} = \theta\Delta_{k,l}$,  $\epsilon_{k,l+1} = \theta\epsilon_{k,l}$, \\$l \leftarrow l + 1$ and go back to Step 1. Otherwise, set $\nu_{k+1} = \nu_k$ and $\sigma_{k+1} = \sigma_k$.\\}
\item[\textbf{Step 6.}]{If $\nu_{k+1} < \nu_\text{opt}$, then terminate. Otherwise, set $x_{k+1} = x_k + d_{k,l}$, \\ $\epsilon_{k+1,0}  = \gamma_\epsilon\nu_{k+1}$, $\Delta_{k+1,0} = \gamma_\Delta\nu_{k+1}$, $k\leftarrow k+1$, $l \leftarrow 0$ and go back to Step 1.}
\end{itemize}
\label{vugs-alg}
\end{algorithm}

In order to guarantee the global convergence of the method, we suppose, from now on, the following assumption.
\begin{assumption}\label{assump1}
For every $k\in \mathbb{N}$, the matrix $H_k\in \mathbb{R}^{n\times n}$ is symmetric positive definite and there exist positive real numbers $\underline{\varsigma}$ and $\overline{\varsigma}$ such that
\begin{equation*}
\underline{\varsigma}\|d\|^2 \leq d^TH_kd \leq \overline{\varsigma}\|d\|^2\text,~~\text{for all}~ d\in\mathbb{R}^n\text.
\end{equation*}
\end{assumption}

\begin{footnotesize}
\begin{center}
\begin{tabular}{|ll|}
\hline
\multicolumn{2}{|c|}{\textbf{Glossary of Notation}} \\ \hline
$k$: outer iteration counter                             & $\nu_k$: optimality certificate               \\
$l$: inner iteration counter                             & $\nu_\text{opt}$: optimality certificate tolerance  \\
$x_k$: current iterate                                   & $\varrho$ and $\delta$: constants for updating $\nu_k$\\
$m$: number of sampled points                            & $\epsilon_{k,l}$: related to the current sampling size \\
$\gamma_\Delta$: constant related to the trust region   & $\Delta_{k,l}$: current trust-region size            \\
$\gamma_\epsilon$: constant related to the sampling size  & $\theta$: reduction factor for $\epsilon_{k,l}$ and $\Delta_{k,l}$ \\
$\rho$: parameter of step acceptance              & $\sigma_k$: power related to the sampling size \\
\hline
\end{tabular}
\end{center}
\end{footnotesize}

%\begin{remark}
%One can reason that, by the way GraFuS was designed, the trust-region size can be dramatically reduced near a stationary point for $f$ that is not a local minimizer, which could compromise the efficiency of the method to solve the problem. As it was already remarked at the introduction of this study, the idea behind GraFuS is to accelerate the local convergence speed of the GS method, and therefore, the method was planned to run only near a neighborhood of a stationary point that has already been obtained by the GS method. Therefore, when the GS method approaches a local minimizer of $f$, GraFuS will speed up the convergence to such a point.     
%\end{remark}

The updating procedure of the matrices $H_k$ is a delicate matter, since a bad sampling at one single iteration might damage some required properties for the convergence theory. For that reason, we give a detailed explanation of how one may update $H_k$ properly in Subsection~\ref{subsec:Hk_update}.

\section{Convergence} Before we proceed with the convergence analysis, we should state a property for the functions $\phi_i$ that define $f$. It is a common assumption when we are dealing with nonsmooth functions of the kind defined in~\eqref{eq:minimax-prob}, cf.~\cite{DSS09,MIS99}. Considering that 
\begin{equation*}
\set I(x) := \{i~|~ \phi_i(x) = f(x)\}\text,
\end{equation*}
the required hypothesis follows.
\begin{assumption}\label{assump:affine}
For all $x\in \mathbb{R}^n$ with $|\set I(x)| \geq 2$, the gradients $\{\nabla \phi_i(x)\}_{i\in \set I(x)}$ compose an affinely independent set, that is, 
\begin{equation*}
\sum_{i\in \set I(x)} \alpha_i \nabla \phi_i(x) = 0\text{~~and~~} \sum_{i\in \set I(x)} \alpha_i = 0~~\Longleftrightarrow~~
\alpha_i = 0\text,~~\text{for all}~ i\in \set I(x)\text.
\end{equation*}  
\end{assumption}

\begin{remark}\label{remark:1}
It is worth pointing out that Assumption~\ref{assump:affine} can be viewed as a way to guarantee that, for any fixed $j\in \set I(x)$, the set 
\begin{equation*}
\{\nabla\phi_i(x) - \nabla\phi_j(x)\}_{{i\in \set I(x)\setminus\{j\}}}  
\end{equation*}
is linearly independent for all $x\in \mathbb{R}^n$ with $|\set I(x)| \geq 2$ (the proof is provided in Lemma~\ref{lemma:enough-points} below). This association will be of great importance for both the global and the local convergence results. 

Additionally, if $x_*$ is a local minimizer for $f$, Assumption~\ref{assump:affine} also gives us that there exists only one possible convex combination of the gradients $\nabla\phi_i(x_*)$, with $i\in \set I(x_*)$, that generates the null vector.
\end{remark}

\subsection{Global convergence}
First, we present a technical lemma guaranteeing that at most $n+1$ functions will assume the maximum of $f$ at a fixed point $x\in \mathbb{R}^n$. In addition, we prove that, for each $\phi_j$, with $j\in \set I(x)$, there is a sufficiently small open set such that $\phi_j$ strictly assumes the maximum value at this specific set.
\begin{lemma}\label{lemma:enough-points}
Under Assumption~\ref{assump:affine}, let $x$ be any point in $\mathbb{R}^n$ and $j$ be any fixed index in $\set I(x)$. Then, $|\set I(x)| \leq n+1$. Moreover, there exists $\overline\epsilon > 0$ such that for all $\epsilon \in (0,\overline\epsilon)$, we can find a set $\set C_j(x,\epsilon)\subset \set B(x,\epsilon)$ with $\inte (\set C_j(x,\epsilon)) \neq \emptyset $, for which $x\notin \set C_j(x,\epsilon)$ and
\begin{equation*}
\phi_j(x^j) > \underset{i\neq j}{\max_{1\leq i\leq p}}~\phi_i(x^j)\text,~~\text{for all~} x^j \in \set C_j(x,\epsilon)\text. 
\end{equation*} 
\end{lemma} 
\begin{proof}First, let us prove that $|\set I(x)| \leq n+1$. If $|\set I(x)| = 1$, the statement trivially holds. Therefore, we assume that $|\set I(x)| \geq 2$. Besides, we suppose without any loss of generality that $\set I(x) = \{1,\hdots,r\}$. Then, let $\alpha_2,\hdots,\alpha_{r}\in \mathbb{R}$ be any real numbers such that
\begin{equation*}
\sum_{i=2}^{r} \alpha_i \left (\nabla \phi_i(x) - \nabla \phi_1(x) \right) = 0\text. 
\end{equation*}
Then, it follows that
\begin{equation*}
-\left(\sum_{i=2}^{r}\alpha_i\right) \nabla \phi_1(x) + \sum_{i=2}^{r}\alpha_i\nabla \phi_i(x) = 0\text,  
\end{equation*}
and, by Assumption~\ref{assump:affine}, we have $\alpha_2=\hdots=\alpha_{r}=0$. Consequently,
\begin{equation*}
\set A:=\{\nabla \phi_i(x) - \nabla \phi_1(x)\}_{i\in \set I(x)\setminus{\{1\}}}
\end{equation*}  
forms a linearly independent set. So, $|\set A|\leq n$, which implies that $|\set I(x)|\leq n+1$.

Now, for the other result, we also have that, if $|\set I(x)| = 1$, then the proof is straightforward by a continuity argument. So, let us suppose that $|\set I(x)| \geq 2$ and $\set I(x) = \{1,\hdots,r\}$. By Assumption~\ref{assump:affine}, given a fixed $s\in \set I(x)$ and any $j\in \set I(x)$ with $j\neq s$, we have that $v_j := \nabla \phi_j(x) - \nabla \phi_s(x)$ cannot be written as a linear combination of $\{ v_i~|~ i\in \set I(x)\text,~i\neq j\}$ (to see this, just use the same arguments that we have used to prove $|\set I(x)| \leq n+1$ and notice that the set formed by the vectors $v_j$'s is linearly independent). Thus, it is possible to find a unitary $d_j\in \mathbb{R}^n$ such that $v_j^Td_j > 0$ and   
$$ v_i^Td_j = 0\text,~~i\neq j\text{~~with~~}i\in \set I(x)\text.\footnotemark\footnotetext{For example, setting $s_j$ as the orthogonal projection of $v_j$ over the hyperplane generated by $\{ v_i~|~ i\in \set I(x)\text,~i\neq j\}$, one can consider $d_j = (v_j - s_j)/\|v_j - s_j\|$.}$$
Consequently, it follows that $ \nabla \phi_j(x)^Td_j > \nabla \phi_s(x)^Td_j$ and
$$ \nabla \phi_i(x)^Td_j = \nabla \phi_s(x)^Td_j\text,~~i\neq j\text{~~with~~}i\in \set I(x)\text.$$
So, since $\phi_i\in C^2$, for all $i\in \set I(x)$, we have that for all fixed $w_j\in \mathbb{R}^n$ it follows that
\begin{equation*}
\begin{split}
\phi_i(x+\epsilon (d_j+w_j)) & = \phi_i(x) + \epsilon \nabla \phi_i(x)^T(d_j+w_j) + O(\epsilon^2)\text,~~i\in \set I(x)\text,~~ i\neq j, \\
\phi_j(x+\epsilon (d_j+w_j)) & = \phi_j(x) + \epsilon \nabla \phi_j(x)^T(d_j+w_j) + O(\epsilon^2)\text.
\end{split}
\end{equation*}
Now, subtracting the first equation above from the second one and dividing the result by $\epsilon$, we obtain, for all $i\in \set I(x)$ with $i\neq j$, that
\begin{equation*}
\begin{split}
\frac{\phi_j(x+\epsilon (d_j+w_j)) - \phi_i(x+\epsilon (d_j+w_j))}{\epsilon} ={ }& \nabla \phi_j(x)^T(d_j+w_j) \\ 
& - \nabla \phi_i(x)^T(d_j+w_j) + O(\epsilon)\text.
\end{split} 
\end{equation*}
Consequently, supposing that
\begin{equation*}
w_j\in \set B\left(0, \delta(x) \right)\subset \mathbb{R}^n\text,
\end{equation*}
where
\begin{equation}\label{eq:delta-x}
\delta(x) := \min_{\underset{i\neq j}{i\in \set I(x)}}\left\{\frac{[\nabla \phi_j(x)-\nabla \phi_i(x)]^Td_j}{2\|\nabla \phi_j(x)-\nabla \phi_i(x)\|}\right\}>0\text,
\end{equation}
we must have, for all $i\in \set I(x)$ with $i\neq j$, that
\begin{equation*}
\begin{split}
\frac{\phi_j(x+\epsilon (d_j+w_j)) - \phi_i(x+\epsilon (d_j+w_j))}{\epsilon} ={ }& [\nabla \phi_j(x) - \nabla \phi_i(x)]^Td_j \\
& + [\nabla \phi_j(x) - \nabla \phi_i(x)]^Tw_j + O(\epsilon) \\
\geq{ }& [\nabla \phi_j(x) - \nabla \phi_i(x)]^Td_j \\ 
& - \|\nabla \phi_j(x) - \nabla \phi_i(x)\|\| w_j \| + O(\epsilon)\\
\geq{ }&\frac{[\nabla \phi_j(x) - \nabla \phi_i(x)]^Td_j}{2} + O(\epsilon)\text. 
\end{split} 
\end{equation*}
From the inequality above and noticing that $[\nabla \phi_j(x) - \nabla \phi_i(x)]^Td_j > 0$, for all $i\in \set I(x)$ with $i\neq j$, it is possible to find $\epsilon_j > 0$ small enough such that for all $\epsilon \in (0,\epsilon_j)$ the following relation holds
\begin{equation*}
\phi_j(x+\epsilon (d_j+w_j)) > \phi_i(x+\epsilon (d_j+w_j))\text,~~ i\in \set I(x) \text,~~i\neq j\text.
\end{equation*}
To complete the proof, notice that the functions $\phi_i$ are continuous, and therefore, it is possible to find $\tilde{\epsilon}>0$ such that for all $y\in \set B(x,\tilde{\epsilon})$ the following holds
\begin{equation*}
\phi_a(y) > \phi_b(y)\text,~~a\in \set I(x)\text,~~b\notin \set I(x)\text.
\end{equation*}
So, setting $\overline\epsilon:=\min\{\epsilon_1,\hdots,\epsilon_r,\tilde{\epsilon}\}$ and choosing $\epsilon \in (0,\overline\epsilon)$, we have that the set
\begin{equation*}
\set C_j(x,\epsilon):=\left\{ x + \tau(d_j+w_j)~|~ 0<\tau<\epsilon/2\text,~~w_j\in \set B\left(0, \delta(x) \right)\text,~~j\in \set I(x) \right\}\text,
\end{equation*}
where $\delta(x)$ is the value defined in~\eqref{eq:delta-x}, satisfies the properties previously claimed.
\end{proof}

From the above result, we can see that, for any $\epsilon > 0$ (even when $\epsilon \geq \overline\epsilon$, since in this case we have $\set B(x,\overline\epsilon) \subset \set B(x,\epsilon)$), the following set is not empty
\begin{equation}\label{eq:open-sets}
\set S_j(x,\epsilon):= \inte\left\{ y\in \set B(x,\epsilon)~\big|~ \phi_j(y) > \underset{i\neq j}{\max_{1\leq i\leq p}} \phi_i(y)\right\}\text,~~j\in \set I(x)\text.
\end{equation}
%and the same can be said about 
%\begin{equation}\label{eq:open-set}
%\overline{\set S}(x,\epsilon):= \set S_{i_1}(x,\epsilon)\times \hdots \times \set S_{i_{|\set I(x)|}}(x,\epsilon)\text,
%\end{equation}
%where $i_j\in \set I(x)$.

So, we can proceed with two additional results. They guarantee that GraFuS is well defined, i.e., the algorithm will not cycle forever from Step 5 to Step 1. Specifically, the first result tells us that under a good set of sampled points, it is possible to obtain $\text{Ared} > \rho\text{Pred}$ (the proof of the result is based on ideas from~\cite{ZKL85}). 
\begin{lemma}\label{lemma:bound}
Suppose that Assumptions~\ref{assump1} and~\ref{assump:affine} hold. In Algorithm~\ref{vugs-alg}, consider fixed outer and inner iterations, denoted by $k$ and $l$, respectively. Let $\overline{x}\in \mathbb{R}^n$ be a nonstationary point for the function $f: \mathbb{R}^n\rightarrow \mathbb{R}$, $\rho\in (0,1)$ be a fixed real number and $\set S_j(\overline{x},\epsilon)$ be the set defined in~\eqref{eq:open-sets} for any $\epsilon > 0$. Therefore, there exist $\overline{\Delta}$ and $\overline{\delta}$ strictly greater than zero such that, if the following hypotheses hold 
\begin{itemize}[leftmargin=1cm]
\item[i)]{$x_k \in \set B(\overline{x},\overline{\delta})$;}
\item[ii)]{$0< \Delta_{k,l} < \overline{\Delta}$;}
\item[iii)]{there exist $\overline\epsilon\equiv \overline\epsilon(k,l) > 0$ and $M > 0$ such that 
\begin{itemize}
\item[a)]{for all $j\in \set I(\overline{x})$, we have ${\set S}_j(\overline{x},\overline\epsilon) \subset \set B(x_k,M\cdot\Delta_{k,l})$;}
\item[b)]{for all $j\in \set I(\overline{x})$, there exists $i\in \{1,\hdots,m\}$ such that $x_{k,i}^l\in {\set S}_j(\overline{x},\overline\epsilon)$;}
\item[c)]{for all $i\in \{1,\hdots,m\}$, there exists $j\in \set I(\overline{x})$ such that $x_{k,i}^l\in \set S_j(\overline{x},\overline\epsilon)$,}
\end{itemize}
}
\end{itemize}
then
\begin{equation*}
\text{Ared}_{k,l} > \rho\text{Pred}_{k,l}\text.
\end{equation*}
\end{lemma}
\begin{proof} 
First, we choose $h>0$ as a sufficiently small number such that for all $x\in \set B(\overline{x},h)$, we have
\begin{equation*}
\phi_j(x) > \underset{i\notin \set I(\overline{x})}{\max_{1\leq i \leq p}} \phi_i(x)\text{,~~for all~}j\in \set I(\overline{x})\text.
\end{equation*}
Since $\overline{x}$ is not a stationary point for $f$, we must have that $0\notin \overline{\partial} f(\overline{x})$. Recalling that $\overline{\partial} f(\overline{x})$ is a closed and convex set, it follows by the Hyperplane Separation Theorem~\cite[Section 2.5]{BOV04} that there exist a unitary vector $v\in \mathbb{R}^n$ and a scalar $\tau > 0$ such that 
\begin{equation*}
s^Tv \leq -\tau\text,~~\text{for all~} s\in \overline{\partial} f(\overline{x})\text.
\end{equation*}
Since the generalized directional derivative of $f$ at $\overline{x}$ in the direction $v$ is given by 
\begin{equation*}
f^\circ(\overline{x};v)= \underset{t\downarrow 0}{\limsup_{x\rightarrow \overline{x}}}~\frac{f(x+tv)-f(x)}{t} = \max\{s^Tv~:~s\in \overline{\partial} f(\overline{x})\}\text,
\end{equation*} 
we have that $f^\circ(\overline{x};v)\leq -\tau$. Thus, there exist $\overline{\Delta}\in (0,h)$ and $\overline{\delta}\in (0,h)$ such that for all $x\in \set B\left(\overline{x},\overline{\delta}\right)$ and $\Delta \in (0,\overline{\Delta})$, we have
\begin{equation}\label{eq:sufficient-decrease}
f(x+\Delta v)-f(x) < -\frac{\tau}{2}\Delta\text.
\end{equation}

Now, let us keep this information in mind and proceed with a parallel idea. Let us suppose that the hypotheses $i)$, $ii)$ and $iii)$ hold for $\overline\delta$ and $\overline\Delta$ found above. Then, because the conditions inside $iii)$ ensure a good sampling, we have
\begin{equation}\label{eq:bound-aux1}
\begin{split}
f(x_k)  = & \max_{j\in \set I(\overline{x})}\{\phi_j(x_k)\}\\
       %& = \max_{1\leq i\leq m}\{\phi_i(x_{k,i}^l) + \nabla \phi_i(x_{k,i}^l)^T(x_k-x_{k,i}^l)\} + o(\Delta_{k,l})\\
        = & \max_{1\leq i\leq m}\{f(x_{k,i}^l) + \nabla f(x_{k,i}^l)^T(x_k-x_{k,i}^l)\} + o(\Delta_{k,l}) \\
        & \text{(notice that~}x_{k,i}^l\in \set B(x_k,M\cdot\Delta_{k,l})\text{)}
\end{split}
\end{equation}       
and
\begin{equation*}
\begin{split}       
f(x_k+d_{k,l}) = & \max_{j\in \set I(\overline{x})}\{\phi_j(x_k+d_{k,l})\}\\
       %& = \max_{1\leq i\leq m}\{\phi_i(x_{ki}^l) + \nabla \phi_i(x_{k,i}^l)^T(x_k+d_{k,l}-x_{k,i}^l)\} + o({\Delta_{k,l}})\\
    = & \max_{1\leq i\leq m}\{f(x_{k,i}^l) + \nabla f(x_{k,i}^l)^T(x_k+d_{k,l}-x_{k,i}^l)\} + o({\Delta_{k,l}})\\
      &  \text{(notice that~}x_{k,i}^l\in \set B(x_k,M\cdot\Delta_{k,l})\text{~and that~}\|d_{k,l}\|_\infty\leq \Delta_{k,l}\text{).}
\end{split}
\end{equation*}
So, we have $\text{Ared}_{k,l} = f(x_k) - f(x_k+d_{k,l}) = \text{Pred}_{k,l} + o({\Delta}_{k,l})$. Consequently, to prove the statement, we just need to show that $\Delta_{k,l} = O(\text{Pred}_{k,l})$, since we would have, for any $\eta = (1-\rho) \in (0,1)$, a sufficiently small $\overline{\Delta}>0$ such that   
\begin{equation*}
\text{Ared}_{k,l} - \text{Pred}_{k,l} = o(\Delta_{k,l}) > -\eta\text{Pred}_{k,l}\text,
\end{equation*}
which yields that $\text{Ared}_{k,l} > (1-\eta)\text{Pred}_{k,l} = \rho\text{Pred}_{k,l}$. So, to show that such a condition holds, we define
\begin{equation*} 
\hat{z} := \max_{1\leq i\leq m}\{f(x_{k,i}^l) + \nabla f(x_{k,i}^l)^T(x_k+\Delta_{k,l} v-x_{k,i}^l)\}\text.
\end{equation*}
Notice that, by the same reasoning used before, we have
\begin{equation}\label{eq:bound-aux2}
\hat{z} = f(x_k + \Delta_{k,l}v) + o(\Delta_{k,l})\text.
\end{equation}
Moreover, since $(d_{k,l},z_{k,l})$ is the solution of the quadratic programming problem at Step 2, we have $z_{k,l}\leq \hat{z} + o(\Delta_{k,l})$, and hence,
\begin{equation*}
\text{Pred}_{k,l} \geq \max_{1\leq i\leq m} \{f(x_{k,i}^l)+\nabla f(x_{k,i}^l)^T(x_k - x_{k,i}^l)\} - \left(\hat{z} + \frac{\Delta_{k,l}^2}{2}v^TH_kv\right) + o(\Delta_{k,l})\text.
\end{equation*}
Consequently, recalling~\eqref{eq:bound-aux1} and~\eqref{eq:bound-aux2}, it yields that
\begin{equation*}
\begin{split}
\text{Pred}_{k,l} & \geq f(x_k)-f(x_k+\Delta_{k,l}v) + o(\Delta_{k,l})\\
& > \frac{\tau}{2}\Delta_{k,l} + o(\Delta_{k,l})\text,
\end{split}
\end{equation*}
where the last inequality comes from~\eqref{eq:sufficient-decrease}. Therefore, if $\overline{\Delta}$ is small enough, we obtain the desired result.
\end{proof}

With the above result, we present the following lemma, which states that if GraFuS is at an iteration $k$ and $x_k$ is not a stationary point for $f$, then the index $l$ of the inner iteration has an upper limit (with probability one).
\begin{lemma}\label{Lemma:l_k}
Suppose that Assumptions~\ref{assump1} and~\ref{assump:affine} hold. Moreover, for an iteration~$k$, assume that $x_k$ is not a stationary point for $f$. Then, with probability one, there exists $\overline{l}\in \mathbb{N}$ such that the indices of the inner iterations satisfy $l\leq \overline{l}$.
\end{lemma}
\begin{proof}
Let us assume, for contradiction, that such $\overline{l}$ does not exist, i.e., $l\rightarrow \infty$ at the iteration $k$. Consequently, we must have, for all $l\in \mathbb{N}$, that 
\begin{equation*}
\|H_k^{-1}G_{k,l}\lambda_{k,l}\| \geq \nu_k
\end{equation*}
and $\text{Ared}_{k,l} \leq \rho\text{Pred}_{k,l}$. Additionally, by the way we have designed our algorithm, we see that
\begin{equation*}
\epsilon_{k,l} = \frac{\gamma_\epsilon}{\gamma_\Delta}\Delta_{k,l}\text,~~\text{for all~} k,l\in \mathbb{N}\text,
\end{equation*}
and, by the contradiction hypothesis, the following holds: $\Delta_{k,l}\rightarrow 0$ as $l\rightarrow \infty$.

Therefore, setting $\overline{x}:= x_k$ in Lemma~\ref{lemma:bound}, it is straightforward to see that at some $\tilde{n}\in \mathbb{N}$, if $l\geq \tilde{n}$, then hypotheses $i)$ and $ii)$ of Lemma~\ref{lemma:bound} are valid. Moreover, considering $\overline{\epsilon}:=(\epsilon_{k,l})^{\sigma_k}$ and $M:= \max\{\gamma_\epsilon\gamma_\Delta^{-1},\gamma_\epsilon^2\gamma_\Delta^{-2}\}$ for a fixed inner iteration $l$, we will satisfy hypothesis $iii)$ item $a)$ of Lemma~\ref{lemma:bound}. Therefore, if at this specific inner iteration $l$ we do not have $\text{Ared}_{k,l} > \rho\text{Pred}_{k,l}$, it is due to the fact that we did not sample the points properly, i.e, the items $b)$ and/or $c)$ of hypothesis $iii)$ were not fulfilled. So, since $l\rightarrow \infty$ by the contradiction hypothesis we have made, it is also true that the next inner iteration will not satisfy items $b)$ and/or $c)$ and so on. We claim that this behavior has probability zero to occur.

Indeed, let us assume a fixed $j\in \set I(x_k)$ and notice that, by the way we have defined $d_j$ and $\set C_j(x_k,(\epsilon_{k,l})^{\sigma_k})$ in the proof of Lemma~\ref{lemma:enough-points}, we have that (for $(\epsilon_{k,l})^{\sigma_k}$ sufficiently small) $\set B_j^{k,l} \subset \set C_j(x_k,(\epsilon_{k,l})^{\sigma_k})$, where
\begin{equation*}
\set B_j^{k,l} := \set B\left(x_k + \frac{(\epsilon_{k,l})^{\sigma_k}}{4}d_j, \frac{(\epsilon_{k,l})^{\sigma_k}}{8}\min_{\underset{i\neq j}{i\in \set I(x_k)}}\left\{\frac{[\nabla \phi_j(x_k)-\nabla \phi_i(x_k)]^Td_j}{2\|\nabla \phi_j(x_k)-\nabla \phi_i(x_k)\|}\right\} \right)\text.
\end{equation*} 
Consequently, the volume of $\set B_j^{k,l}$ in $\mathbb{R}^n$ is given by
\begin{equation*}
\vol\left(\set B_j^{k,l}\right) = \frac{\pi^{n/2}}{\Gamma(n/2+1)}\left(\min_{\underset{i\neq j}{i\in \set I(x_k)}}\left\{\frac{[\nabla \phi_j(x_k)-\nabla \phi_i(x_k)]^Td_j}{2\|\nabla \phi_j(x_k)-\nabla \phi_i(x_k)\|}\right\} \right)^n\left(\frac{(\epsilon_{k,l})^{\sigma_k}}{8}\right)^n\text,
\end{equation*}
where $\Gamma$ is the Gamma function~\cite{HUB82}. On the other hand, it follows that
\begin{equation*}
\vol (\set B(x_k,(\epsilon_{k,l})^{\sigma_k})) = \frac{\pi^{n/2}}{\Gamma(n/2+1)}\left((\epsilon_{k,l})^{\sigma_k}\right)^n\text. 
\end{equation*}
Therefore, since the sampled points are chosen in $\set B(x_k,(\epsilon_{k,l})^{\sigma_k})$ and 
\begin{equation*}
\set B_j^{k,l}\subset \set C_j(x_k,(\epsilon_{k,l})^{\sigma_k}) \subset \set S_j(x_k,(\epsilon_{k,l})^{\sigma_k})\text,
\end{equation*} 
we must have, for all $i\in \{1,\hdots,m\}$, that the conditional probability
\begin{equation*}
\proj (x_{k,i}^l\in \set S_j(x_k,(\epsilon_{k,l})^{\sigma_k})~|~x_{k,i}^l\in \set B(x_k,(\epsilon_{k,l})^{\sigma_k})) = \frac{\vol(\set S_j(x_k,(\epsilon_{k,l})^{\sigma_k}))}{\vol (\set B(x_k,(\epsilon_{k,l})^{\sigma_k}))} \\ 
\end{equation*}
must be greater than the following strictly positive number
\begin{equation*}
\frac{1}{8^n}\left(\min_{\underset{i\neq j}{i\in \set I(x_k)}}\left\{\frac{[\nabla \phi_j(x_k)-\nabla \phi_i(x_k)]^Td_j}{2\|\nabla \phi_j(x_k)-\nabla \phi_i(x_k)\|}\right\} \right)^n\text.
\end{equation*} 
With this inequality, we conclude that the probability of the items $b)$ and $c)$ of hypothesis $iii)$ to occur simultaneously is strictly positive and does not depend on $l$. Therefore, the probability of $l\rightarrow \infty$ is zero, which concludes the proof. 
\end{proof}

We are close to reach the convergence theorem of GraFuS. For that goal, we need to prove two additional technical lemmas. Furthermore, to have a clearer proof, from now on we will denote by $\overline{l}_k$ the largest value of the index $l$ at the iteration $k$, established by Lemma~\ref{Lemma:l_k}.
\begin{lemma}\label{lemma:norm-conv}
Let us consider the GraFuS algorithm under Assumptions~\ref{assump1} and~\ref{assump:affine}. If there exists an infinite index set $\tilde{\set K}\subset \mathbb{N}$ such that $\text{Pred}_{k,\overline{l}_k}/\Delta_{k,\overline{l}_k} \underset{k\in \tilde{\set K}}{\rightarrow} 0$, then $ \|G_{k,\overline{l}_k}\lambda_{k,\overline{l}_k}\| \underset{k\in \tilde{\set K}}{\rightarrow} 0$.
\end{lemma} 
\begin{proof}
First, notice that the quadratic programming problem presented in~\eqref{eq:qp-max} satisfies the Slater's condition. Indeed, if one considers $d_k = 0$ and $z_k = \max\{\tilde f_{k}\} + 1$ in~\eqref{eq:qp-max}, then we see that all inequalities are strictly satisfied. Thus, since the problem is also convex, we can guarantee that the quadratic programming problem satisfies strong duality. So, we have
\begin{equation*}
\begin{split}
z_{k,\overline{l}_k} + \frac{1}{2}d_{k,\overline{l}_k}^TH_kd_{k,\overline{l}_k} = &~~\lambda_{k,\overline{l}_k}^T\tilde{f}_{k,\overline{l}_k} \\
& - \frac{1}{2}\left(G_{k,\overline{l}_k}\lambda_{k,\overline{l}_k} + \omega_{k,\overline{l}_k}\right)^TH_k^{-1}\left(G_{k,\overline{l}_k}\lambda_{k,\overline{l}_k} + \omega_{k,\overline{l}_k}\right)\\
& - \Delta_{k,\overline{l}_k}\|\omega_{k,\overline{l}_k}\|_1\text.
\end{split} 
\end{equation*}
Thus, defining 
\begin{equation}\label{eq:alpha}
\alpha_k := \frac{1}{2}\left(G_{k,\overline{l}_k}\lambda_{k,\overline{l}_k} + \omega_{k,\overline{l}_k}\right)^TH_k^{-1}\left(G_{k,\overline{l}_k}\lambda_{k,\overline{l}_k} + \omega_{k,\overline{l}_k}\right) + \Delta_{k,\overline{l}_k}\|\omega_{k,\overline{l}_k}\|_1\text,
\end{equation}
it yields
\begin{equation*}
\begin{split}
\lambda_{k,\overline{l}_k}^T\tilde{f}_{k,\overline{l}_k} - \alpha_k  = z_{k,\overline{l}_k} + \frac{1}{2}d_{k,\overline{l}_k}^TH_kd_{k,\overline{l}_k} & \Rightarrow \alpha_k = \lambda_{k,\overline{l}_k}^T\tilde{f}_{k,\overline{l}_k} \\ 
&~~~~~~~~~~~- \left(z_{k,\overline{l}_k} + \frac{1}{2}d_{k,\overline{l}_k}^TH_kd_{k,\overline{l}_k}\right)\\
& \Rightarrow \alpha_k \leq \text{Pred}_{k,\overline{l}_k}\\
& \text{~~~~(since $\lambda_{k,\overline{l}_k}\geq 0$ and $e^T\lambda_{k,\overline{l}_k} = 1$)}\\
& \Rightarrow \frac{\alpha_k}{\Delta_{k,\overline{l}_k}} \leq \frac{\text{Pred}_{k,\overline{l}_k}}{\Delta_{k,\overline{l}_k}}\\
& \Rightarrow \frac{\alpha_k}{\Delta_{k,\overline{l}_k}} \underset{k\in \tilde{\set K}}{\rightarrow} 0\text.
\end{split}
\end{equation*}
Consequently, by Assumption~\ref{assump1} and~\eqref{eq:alpha}, we obtain $\|G_{k,\overline{l}_k}\lambda_{k,\overline{l}_k}\|\underset{k\in \tilde{\set K}}{\rightarrow} 0$. 
\end{proof}

Finally, we present the last result before our main statement of the global convergence analysis.

\begin{lemma}\label{lemma:nu-zero}
Suppose that Assumptions~\ref{assump1} and~\ref{assump:affine} hold and GraFuS has generated an infinite sequence $\{x_k\}\subset \mathbb{R}^n$. Moreover, assume that there exists a cluster point $\overline{x}$ of this sequence that is a stationary point for $f$. Then, with probability one, the sequence $\{\nu_k\}$ must converge to zero.
\end{lemma}
\begin{proof}
By hypothesis, we have that $0\in \overline\partial f(\overline{x})$. Moreover, all the functions that comprise $f$ are of class $C^2$. So, it is possible to find, for any given $\delta_1,\delta_2>0$, nonempty and open sets $\set X_1,\hdots,\set X_m \subset \set D$ and a fixed vector $\overline\lambda\in \mathbb{R}^m$ satisfying $\overline\lambda \geq 0$ and $e^T\overline\lambda = 1$ such that
\begin{equation}\label{eq:nu-zero1} 
\set X_j \subset \set B(\overline{x},\delta_1)\text{, for all }j\in \{1,\hdots,m\}\text,
\end{equation} 
and
\begin{equation}\label{eq:nu-zero2}
\left\|\sum_{j=1}^m\overline\lambda_j\nabla f(x_j)\right\| \leq \delta_2\text{, for all}~ (x_1,\hdots,x_m)\in \set X_1\times\cdots\times\set X_m\text. 
\end{equation}

By contradiction, let us assume that $\{\nu_k\}$ does not go to zero, i.e., there exists $\overline\nu > 0$ such that $\nu_k = \overline\nu$ for all $k\in \mathbb{N}$ sufficiently large. This condition yields that $\epsilon_{k,0} = \gamma_\epsilon\overline\nu$ and $\Delta_{k,0} = \gamma_\Delta\overline\nu$ for all $k\in \mathbb{N}$ large enough. Moreover, noticing that Lemma~\ref{lemma:norm-conv} also holds if we consider the inner iteration $0$ instead of $\overline{l}_k$, we have that
\begin{equation}\label{eq:aux-Pred}
\frac{\text{Pred}_{k,0}}{\Delta_{k,0}} \geq \mu\text,
\end{equation}
for some $\mu > 0$. Otherwise, $\|G_{k,0}\lambda_{k,0}\|$ would go to zero, implying that $\nu_k$ would also go to zero.

Defining $\set K$ as an infinite index set such that $\{x_k\}_{k\in \set K}$ converges to $\overline{x}$, it is possible to find $\delta_1$ small enough such that~\eqref{eq:nu-zero1} holds and
\begin{equation}\label{eq:nu-zero3}
\set{X}_j \subset \set B(x_k,(\epsilon_{k,0})^{\sigma_k})\text{, for all }j\in \{1,\hdots,m\}\text{ and }k\in \set K\text{ large enough}\text.
\end{equation}
So, let us suppose that for some $k\in \set K$ sufficiently large, we have $x_{k,j}^0 \in \set X_j$, for all $j\in \{1,\hdots,m\}$. Then, considering $\lambda_{k,0}$ and $\omega_{k,0}$ the solutions obtained at Step 2, we must have
\begin{equation*}
\begin{split}
\overline\lambda^T\tilde f_{k,0} - \frac{1}{2}\overline\lambda^TG_{k,0}H_k^{-1}G_{k,0}\overline\lambda \leq{} & \lambda_{k,0}^T\tilde{f}_{k,0} \\
& - \frac{1}{2}\left(G_{k,0}\lambda_{k,0} + \omega_{k,0}\right)^TH_k^{-1}\left(G_{k,0}\lambda_{k,0} + \omega_{k,0}\right)\\
& - \Delta_{k,0}\|\omega_{k,0}\|_1\text.
\end{split}
\end{equation*} 
Adding 
\begin{equation*}
\max_{1\leq j\leq m}\left\{\left(\tilde f_{k,0}\right)_j\right\} 
\end{equation*}
to both sides of the inequality that comes from multiplying the previous one by $(-1)$ and considering the strong duality of the quadratic problem that is solved in Step 2, we have
\begin{equation*}
\text{Pred}_{k,0} \leq \max_{1\leq j\leq m}\left\{\left(\tilde f_{k,0}\right)_j\right\} - \overline\lambda^T\tilde f_{k,0} + \frac{1}{2}\overline\lambda^TG_{k,0}H_k^{-1}G_{k,0}\overline\lambda\text.
\end{equation*}
Since $f(\overline{x}) = \phi_i(\overline{x})$, for any $i\in \set I(\overline x)$, it is possible to select a sufficiently small $\delta_1$, such that~\eqref{eq:nu-zero1}, \eqref{eq:nu-zero3} and
\begin{equation*}
\left|\max_{1\leq j\leq m}\left\{\left(\tilde f_{k,0}\right)_j\right\} - \overline\lambda^T\tilde f_{k,0}\right| \leq \mu\frac{\gamma_\Delta\overline\nu}{4}\text,
\end{equation*} 
are valid for any $k\in \set K$ large enough. Moreover, by Assumption~\ref{assump1}, it is possible to choose $\delta_2$ sufficiently small such that~\eqref{eq:nu-zero2} holds and 
\begin{equation*}
\frac{1}{2}\overline\lambda^TG_{k,0}H_k^{-1}G_{k,0}\overline\lambda \leq \mu\frac{\gamma_\Delta\overline\nu}{4}\text,
\end{equation*}
for any $k\in \set K$ large enough.

As a result, there are $\delta_1,\delta_2 > 0$ sufficiently small and $k\in \set K$ sufficiently large, such that, if $x_{k,j}^0\in \set X_j$, for all $j\in \{1,\hdots,m\}$, we have
\begin{equation*}
\frac{\text{Pred}_{k,0}}{\Delta_{k,0}} \leq \mu\frac{\gamma_\Delta\overline\nu}{2\Delta_{k,0}} = \frac{\mu}{2}\text.
\end{equation*}

Since we have supposed that $\nu_k$ does not go to zero, it implies that GraFuS never samples in the nonempty and open set $\set X_1\times\cdots\times\set X_m$ during the iterations $k\in \set K$, since, otherwise, we would have a contradiction with~\eqref{eq:aux-Pred}. This is an event that has probability zero to occur. Therefore, with probability one, the sequence $\{\nu_k\}$ must converge to zero.   
\end{proof}

Now, we present the main result of this subsection. Using the result below, we can prove the global convergence of GraFuS.
\begin{theorem}\label{theo:global}
Under Assumptions~\ref{assump1} and~\ref{assump:affine}, suppose that $f$ has bounded level sets and GraFuS produces an infinite sequence $\{x_k\}$ with $\nu_\text{opt} = 0$. Then, with probability one, the sequence $\{\nu_k\}$ converges to zero.
\end{theorem}
\begin{proof}
We split the proof in two complementary cases:
\begin{itemize}
\item[$i)$]{There are an infinite set of indices $\set K_1\subset \mathbb{N}$ and a real number $\overline{\epsilon}>0$ such that $\epsilon_{k,\overline{l}_k}\geq \overline{\epsilon}$ for all $k\in \set K_1$.}
\item[$ii)$]{The sampling radius along the iterations satisfy $\displaystyle\epsilon_{k,\overline{l}_k} \underset{k\in \mathbb{N}}{\rightarrow} 0$.}
\end{itemize}
Initially, let us suppose that case $i)$ holds. So, noticing that $\epsilon_{k,\overline{l}_k}\leq \gamma_\epsilon\nu_k$, for all $k\in \mathbb{N}$, and that $\{\nu_k\}$ is a monotonically decreasing sequence, we see clearly that there must exist $\overline{\nu}$ such that $\nu_k\geq \overline{\nu}$, for all $k\in \mathbb{N}$. Therefore, by the way GraFuS was designed, it means that for a sufficiently large index $k$ and any inner iteration $l$, the inequality
\begin{equation*}
\|H_k^{-1}G_{k,l}\lambda_{k,l}\| < \nu_k
\end{equation*}
will never hold, and consequently, the sequence generated by the values $f(x_k)$ will decrease monotonically.  Additionally, we claim that there exists $\mu>0$ such that $\Delta_{k,\overline{l}_k}\mu \leq \text{Pred}_{k,\overline{l}_k}$, for all $k\in \mathbb{N}$. Indeed, if this statement were false, there would exist an infinite set of indices $\tilde{\mathcal K}$ such that 
\begin{equation*}
\text{Pred}_{k,\overline{l}_k}/\Delta_{k,\overline{l}_k}\underset{k\in \tilde{\mathcal K}}{\rightarrow} 0\text.
\end{equation*}
However, by Lemma~\ref{lemma:norm-conv}, it would yield that
\begin{equation*}
\|G_{k,\overline{l}_k}\lambda_{k,\overline{l}_k}\|\underset{k\in \tilde{\mathcal K}}{\rightarrow} 0\text.
\end{equation*}
Therefore, we would have $\nu_k\rightarrow 0$, and consequently, that $\epsilon_{k,\overline{l}_k}\rightarrow 0$, which is a contradiction with case $i)$. Thus, there must exist $\mu>0$ such that $\Delta_{k,\overline{l}_k}\mu \leq \text{Pred}_{k,\overline{l}_k}$, for all $k\in \mathbb{N}$ sufficiently large. Moreover, since 
\begin{equation*}
\epsilon_{k,l} = \frac{\gamma_\epsilon}{\gamma_\Delta}\Delta_{k,l}\text,~~\text{for all } k,l\in \mathbb{N}\text,
\end{equation*}
we see that
$\Delta_{k,\overline{l}_k} \geq (\gamma_\Delta/\gamma_\epsilon) \overline{\epsilon}$, for all $k \in \set K_1$. Consequently, since we have
\begin{equation*}
\text{Ared}_{k,\overline{l}_k} > \rho \text{Pred}_{k,\overline{l}_k}\text,~~\text{for all } k\in \set K_1\text{ sufficiently large,} 
\end{equation*}
we obtain
\begin{equation}\label{eq:f-bound}
f(x_k) - f(x_{k+1}) > \rho\mu\frac{\gamma_\Delta}{\gamma_\epsilon}\overline{\epsilon}\text,~~\text{for all } k\in \set K_1\text{ sufficiently large.} 
\end{equation}
Now, since $f$ has bounded level sets, there must exist an infinite set of indices $\set K_2\subset \set K_1$ such that 
\begin{equation*}
x_k\underset{k\in \set K_2}{\rightarrow} \hat x\text,~~\text{for some~}\hat x\in \mathbb{R}^n\text.
\end{equation*}
So, considering $s_{\set K_2}(k)$ as the index in $\set K_2$ that comes right after $k\in \set K_2$ and recalling that, for the case at hand, it is possible to find a sufficiently large $\hat k\in \set K_2$, where the sequence of function values will be a decreasing sequence for all $k\in \mathbb{N}$ and $k \geq \hat k$, it yields that
\begin{equation*}
\begin{split}
\sum_{k\in \set K_2,k\geq \hat k} \left( f(x_k) - f(x_{k+1}) \right) & \leq  \sum_{k\in \set K_2,k\geq \hat k} \left( f(x_k) - f\left(x_{s_{\set K_2}(k)}\right) \right)\\
& = f\left(x_{\hat k}\right) - f(\hat{x}) <  \infty\text.
\end{split}
\end{equation*}
However, this is a relation that goes against~\eqref{eq:f-bound}. Therefore, the case $i)$ is an impossible event and we must consider case $ii)$.

So, suppose that case $ii)$ holds and, by contradiction, that the sequence $\{\nu_k\}$ does not converge to zero. Again, we must have that the inequality
\begin{equation*}
\|H_k^{-1}G_{k,l}\lambda_{k,l}\| < \nu_k
\end{equation*}
will never hold for $k$ sufficiently large, and consequently, the sequence generated by the values $f(x_k)$ will decrease monotonically. Thus, there must exist at least one cluster point $\overline{x}$ of $\{x_k\}$. Consequently, there is $\tilde{\set K}\subset \mathbb{N}$ such that   
\begin{equation*}
x_k\underset{k\in \tilde{\set K}}{\rightarrow} \overline x\text.
\end{equation*} 
Now, because of Lemma~\ref{lemma:nu-zero}, $\overline{x}$ is not a stationary point for $f$. Then, we choose $\overline\delta,\overline\Delta$ as presented in Lemma~\ref{lemma:bound} for the point $\overline{x}$. Since $\nu_k$ remains bounded away from zero by our assumption and $\epsilon_{k,\overline{l}_k}\rightarrow 0$, we have, by the way we have designed GraFuS, that $\epsilon_{k,\overline{l}_k}$ just keeps going smaller because $\overline{l}_k\rightarrow \infty$. As a consequence, there exist $k',l'\in \mathbb{N}$ such that for all $k\geq k'$ we have 
\begin{equation*}
\Delta_{k,l'} = \tilde\Delta := \left(\theta^{l'}\right)\gamma_\Delta\nu_k < \overline{\Delta}~~\text{and}~~ \epsilon_{k,l'} = \tilde\epsilon:= \left(\theta^{l'}\right) \gamma_\epsilon\nu_k = \frac{\gamma_\epsilon}{\gamma_\Delta}\tilde{\Delta}\text.
\end{equation*}
Moreover, since $\overline{x}$ is a cluster point for the sequence of iterates, we can find $\hat k \geq k'$ such that for all $k\geq \hat k$ and $k\in \tilde{\set K}$, we have 
\begin{equation*}
x_k\in \set B(\overline{x},\min\{\tilde\epsilon^2,\overline\delta\}/4) \subset \set B(\overline{x},\min\{\tilde\epsilon^{\sigma_k},\overline\delta\}/4)
\end{equation*}

So, for all $j\in \set I(\overline{x})$, we have    
\begin{equation*}
x_k\in \set B\left(\overline{x},\min\left\{\tilde\epsilon^{\sigma_k},\overline\delta\right\}/4\right)\text{ and }\set S_j(\overline x,\min\{\tilde\epsilon,\overline\delta\}/4)\subset \set B\left(x_k,\frac{\gamma_\epsilon}{\gamma_\Delta}\tilde\Delta\right)\text.
\end{equation*}
Therefore, the hypotheses $i)$, $ii)$ and $iii)$ item $a)$ of Lemma~\ref{lemma:bound} are all satisfied. Thus, since $\overline{l}_k\rightarrow \infty$, we must have that items $b)$ and/or $c)$ of hypothesis $iii)$ are not satisfied for every $k\geq \hat k$ and $l = l'$. However, this is an event with probability zero of happening, since the sets $\set S_j(\overline x,\min\{\tilde\epsilon,\overline\delta\}/4)$ are open and not empty. Consequently, with probability one, the sequence $\{\nu_k\}$ must converge to zero.
\end{proof}

In the light of the above theorem, the next corollary ensures that GraFuS will find, in a finite number of iterations, an $\epsilon$-stationary point under any given tolerance. Furthermore, it justifies calling $\nu_k$ an optimality certificate.

\begin{corollary}
Under Assumptions~\ref{assump1} and~\ref{assump:affine}, suppose that $f$ has bounded level sets and the parameter value $\nu_\text{opt}$ in GraFuS is strictly positive. Then, with probability one, GraFuS terminates in a finite number of iterations. Moreover, there exists $v\in \mathbb{R}^n$ such that 
\begin{equation*}
v\in \overline\partial_{\tilde\epsilon} f\left(x_{\hat{k}}\right)~~\text{with}~~ \|v\| \leq \tilde\nu :=  \overline{\varsigma}\cdot\nu_{\hat k}\text,
\end{equation*} 
where $\hat{k}$ is the final iteration of GraFuS, $\tilde\epsilon := \gamma_\epsilon\nu_{\hat k}$ and $\overline{\varsigma}$ is the constant presented in Assumption~\ref{assump1}. In other words, $x_{\hat{k}}$ is an $\tilde\epsilon$-stationary point under the tolerance $\tilde\nu$.
\end{corollary}
\begin{proof}
The proof follows immediately from Theorem~\ref{theo:global}, Assumption~\ref{assump1} and by the way GraFuS was designed.
\end{proof}

The next result guarantees that if the sequence $\{x_k\}$ produced by GraFuS is bounded, then we also obtain an asymptotic result.
\begin{corollary}
Under Assumptions~\ref{assump1} and~\ref{assump:affine}, suppose that GraFuS produces an infinite and bounded sequence $\{x_k\}$ with $\nu_\text{opt} = 0$. Then, with probability one, there is at least one cluster point of this sequence such that it is a stationary point for $f$. 
\end{corollary}
\begin{proof}
The result follows immediately from~Theorem~\ref{theo:global}. Notice that replacing the boundness of the level sets by the boundness of $\{x_k\}$ does not invalidate the proof of Theorem~\ref{theo:global}. Additionally, defining the infinite index set
\begin{equation*}
\set K:= \{k\in \mathbb{N}~|~\nu_{k+1}<\nu_k\}\text,
\end{equation*}
we have that, since $\{x_k\}$ is bounded, there must exist a point $\overline x\in \mathbb{R}^n$ and an infinite index set $\tilde{\set K}\subset \set K$ such that 
\begin{equation*}
x_k\underset{k\in \tilde{\set K}}{\rightarrow} \overline x\text.
\end{equation*} 
Therefore, since $\nu_k\rightarrow 0$ and there exists $v_k\in \mathbb{R}^n$ such that
\begin{equation*}
v_k\in \overline\partial_{(\gamma_\epsilon \nu_k)} f(x_k)~~\text{with}~~\|v_k\| \leq \overline\varsigma\cdot \nu_{k}\text{, for any }k\in \set{\tilde{K}}\text,
\end{equation*}
we have the desired result (see item $iii)$ of \cite[Lemma 3.2]{KWL07}), i.e., $0\in \overline\partial f(\overline x)$ with probability one.  
\end{proof}

In the next subsection, we show that, under a good sampling, the method superlinearly either moves to a local minimizer of $f$ or reduces the optimality certificate. For such a goal, our analysis will involve the concept of $\set U$ and $\set V$ spaces.

\subsection{Local convergence}
In this subsection our efforts will be focused in enlightening the role played by the quadratic programming problem~\eqref{eq:qp-max}. In fact, under special circumstances, it is possible to see this quadratic problem as a local approximation of a new optimization problem that involves the smooth functions $\phi_i$. Under this new perspective, we can analyze the local convergence of the proposed method and obtain interesting results. However, since our method has a random nature and a good local information about the function is restricted to a good set of sampled points, it is reasonable to think that a good rate of convergence will not be achieved at every iteration. Therefore, the results presented here will be sustained on hypotheses that guarantee a good sampling. Additionally, the following definition presents key concepts for our analysis (a more general definition can be found in \cite{LEW02}).
\begin{definition}[$\set U$,$\set V$-spaces]
Suppose that $f:\mathbb{R}^n\rightarrow \mathbb{R}$ is the continuous objective function of problem~\eqref{eq:minimax-prob} and $x$ is any point in $\mathbb{R}^n$. Then, we define 
\begin{equation*}
\set U(x):=\{s\in \mathbb{R}^n~|~[\nabla\phi_i(x)-\nabla\phi_j(x)]^Ts = 0\text,~~\forall i,j\in \set I(x)\text,~~i\neq j\}
\end{equation*}
and $\set V(x):= \set U(x)^\perp$ as the smooth and nonsmooth subspaces of $f$ at $x$, respectively. 
\end{definition}

To accomplish the aim of this subsection, under Assumption~\ref{assump:affine}, we start supposing, without any loss of generality, that
\begin{equation*} 
\set I(x_*) = \{1,\hdots,r+1\}\text{,~~for some~} r\leq n\text.
\end{equation*}    
Moreover, we assume that $x_*\in\mathbb{R}^n$ is a local minimizer of the optimization problem presented in~\eqref{eq:minimax-prob} and that $x_*$ is also a strong minimizer for $f$~\cite[Section 5.1]{MIS05}. 
\begin{assumption}\label{assump:strongmin}
The local minimizer $x_*$ of problem~\eqref{eq:minimax-prob} is a strong minimizer, i.e., $0\in \ri\overline{\partial}f(x_*)$ and there exists $\mu > 0$ such that
\begin{equation}\label{eq:hessian-min}
d^T\left(\sum_{i=1}^{r+1}\left(\lambda_*\right)_i\nabla^2\phi_i(x_*)\right)d \geq \mu\|d\|^2\text{,\quad for all~} d\in \set U(x_*)\text, 
\end{equation}
where $\lambda_*\in \mathbb{R}^{r+1}$ is the unique vector such that
\begin{equation*}
\lambda_* \geq 0\text,~~\sum_{i=1}^{r+1} \left(\lambda_*\right)_i = 1~~\text{and}~~\sum_{i=1}^{r+1}\left(\lambda_*\right)_i\nabla\phi_i(x_*) = 0\text.
\end{equation*}     
\end{assumption}

Below, we present our first technical result that will prove helpful for the subsequent statements.

\begin{lemma}\label{lemma:good-sample}
Suppose that Assumptions~\ref{assump1}, \ref{assump:affine} and~\ref{assump:strongmin} hold and $\{x_k\}$ is an infinite sequence generated by GraFuS with $\nu_k\rightarrow 0$ and $x_k\rightarrow x_*$. Then, there exists an infinite index set $\set K\subset \mathbb{N}$ such that, for any fixed $k\in \set K$, the following holds:
\begin{itemize}
\item[$i)$]{for each $j\in \{1,\hdots,m\}$, there exists $i\in \set I(x_*)$, such that 
\begin{equation*}
\nabla f\left(x_{k,j}^{\overline{l}_k}\right) = \nabla \phi_i\left(x_{k,j}^{\overline{l}_k}\right)\text;
\end{equation*}}
\item[$ii)$]{for each $i\in \set I(x_*)$, there exists $j\in \{1,\hdots,m\}$, such that
\begin{itemize}
\item[$a)$]{
$\nabla f\left(x_{k,j}^{\overline{l}_k}\right) = \nabla \phi_i\left(x_{k,j}^{\overline{l}_k}\right)$;
}
\item[$b)$]{
$\big(\lambda_{k,{\overline{l}_k}}\big)_j>0$, i.e., the constraint
\begin{equation*}
f\left(x_{k,j}^{\overline{l}_k}\right) + \nabla f\left(x_{k,j}^{\overline{l}_k}\right)^T\left(x_k + d - x_{k,j}^{\overline{l}_k}\right) \leq z
\end{equation*}
is active at the optimal solution of the quadratic programming that is solved in Step 2.
}\end{itemize}
}
\end{itemize}
\end{lemma}
\begin{proof}
We claim that, for a sufficiently large $\overline k\in \mathbb{N}$, the following infinite index set
\begin{equation}\label{eq:K-set}
\set K:=\{k \in \mathbb{N} ~|~ \nu_{k+1} < \nu_{k}\text,~k\geq \overline k \}
\end{equation}
has the required properties. Indeed, recalling that $\epsilon_{k,0}\rightarrow 0$ (since $\nu_k\rightarrow 0$), $x_k\rightarrow x_*$ and $\phi_i$ are all continuous functions, then, for any large outer iteration, there exists $\set W\subset \mathbb{R}^n$ such that
\begin{equation*}
x_{k,j}^l\in \set W\text{, for all~}j\in \{1,\hdots,m\}\text{ and }l\in \mathbb{N}\text,
\end{equation*}
where $\set W$ is a neighborhood of $x_*$ such that only the functions $\phi_i$, with $i\in \set I(x_*)$, assume the maximum in this set, which gives us $i)$. 

Now, by the way we have designed Step 3, there exists $\hat\lambda\in \mathbb{R}^m$ such that
\begin{equation*}
\|H_k^{-1}G_{k,\overline l_k}\hat\lambda\| < \nu_k\text{, for any }k\in \set K\text.
\end{equation*}  
Additionally, because we assume that $0\in \ri\overline\partial f(x_*)$ and Assumption~\ref{assump:affine} holds, it follows, by~\cite[Remark III.2.1.4]{HUL93}, that
\begin{equation}\label{eq:0-generate}
\sum_{i \in \set I(x_*) } \lambda_i\nabla \phi_i(x_*) = 0 \Rightarrow \lambda_i > 0\text{,~~for all~}i \in \set I(x_*)\text. 
\end{equation}
Therefore, since the functions that comprise $f$ are assumed to be of class $C^2$, it is not possible to have $\nu_k\rightarrow 0$ without having $ii)$, item $a)$. 

Finally, let us suppose for contradiction that there exists $i\in \set I(x_*)$ such that, for any $j\in \{1,\hdots,m\}$ with $\nabla f(x_{k,j}^{\overline{l}_k}) = \nabla \phi_i(x_{k,j}^{\overline{l}_k})$ and $k\in \set K$, we have $(\lambda_{k,\overline{l}_k})_j = 0$. So, since the trust-region constraints are not active for the (outer,~inner) iteration pair $(k,\overline{l}_k)$, whenever $k\in \set K$, and recalling that $\lambda_{k,\overline{l}_k}$ is the optimal solution of 
\begin{equation*}
\begin{split}
\max_{\lambda\in \mathbb{R}^{m}}~~ & \lambda^T\tilde{f}_{k,\overline{l}_k} - \frac{1}{2}\lambda^TG_{k,\overline{l}_k}^TH_k^{-1}G_{k,\overline{l}_k}\lambda\\
\text{s.t.~~} & \lambda^Te = 1 \\
             & \lambda \geq 0\text,
\end{split}
\end{equation*} 
we see, by implication~\eqref{eq:0-generate}, that, if $\overline{k}$ presented in~\eqref{eq:K-set} is sufficiently large, there must exist $M>0$ such that 
\begin{equation*}
M < \frac{1}{2}\lambda_{k,\overline{l}_k}^TG_{k,\overline{l}_k}^TH_k^{-1}G_{k,\overline{l}_k}\lambda_{k,\overline{l}_k}\text{, for all }k\in \set K\text.
\end{equation*}
This implies that $\lambda_{k,\overline{l}_k}$ cannot be the optimal solution, since any $\hat\lambda \in \mathbb{R}^m$, with $\|H_k^{-1}G_{k,\overline{l}_k}\hat\lambda\|<\nu_k$, will give a better function value whenever $k\in \set K$ is large enough (i.e. $\nu_k$ is small enough). In conclusion, we must have that $ii)$, item $b)$, holds.
\end{proof}

Along this subsection, every time we refer to the set written as $\set K$, we are referring to the set $\set K$ defined in~\eqref{eq:K-set}. Moreover, recalling the result obtained above and rearranging properly the sampled points, we can suppose, without any loss of generality, that
\begin{equation*}
\nabla f \left(x_{k,i}^{\overline{l}_k}\right) = \nabla \phi_i \left(x_{k,i}^{\overline{l}_k}\right)\text{ and } \left(\lambda_{k,\overline{l}_k}\right)_i > 0\text{, for all }k\in \set K\text{ and }i\in \set I(x_*)\text.
\end{equation*}
Additionally, for the sake of simplicity, we assume from now on that, for $k\in \set K$, $(\lambda_{k,\overline{l}_k})_i = 0$, if $i \notin \set I(x_*)$. We lead the reader to the Appendix of this study to see that the same local convergence result presented in this subsection can be obtained without this additional assumption. We also stress that when $k\in \set K$, the trust-region constraints are not active for the last inner iteration $\overline{l}_k$.     

So, for any $k\in \set K$, one can rewrite~\eqref{eq:qp-max} as the following optimization problem 
\begin{equation}\label{eq:just-active}
\begin{split}
\min_{\left(d,z\right)\in \mathbb{R}^{n+1}~~} & z + \frac{1}{2}d^TH_kd \\
\text{s.t.~~} & \phi_{i}\left(x_{k,i}^{\overline{l}_k}\right) + \nabla \phi_{i}\left(x_{k,i}^{\overline{l}_k}\right)^T\left(x_k + d - x_{k,i}^{\overline{l}_k}\right) = z_\text,~~1\leq i\leq r+1\text.
\end{split} 
\end{equation}
Alternatively, it can also be viewed as
\begin{equation}\label{eq:equivalent}
\begin{split}
\min_{d\in \mathbb{R}^{n}~~} & \phi_{r+1}\left(x_{k,r+1}^{\overline{l}_k}\right) +  \nabla \phi_{r+1}\left(x_{k,r+1}^{\overline{l}_k}\right)^T\left(x_k + d - x_{k,r+1}^{\overline{l}_k}\right) + \frac{1}{2}d^TH_kd \\
\text{s.t.~~} & \tilde{\Phi}_k + \tilde{J}_kd = 0\text,
\end{split} 
\end{equation}
where $\tilde{\Phi}_k \in \mathbb{R}^{r}$ with
\begin{equation*}
\begin{split}
(\tilde{\Phi}_k)_i := &~ 
\phi_{i}\left(x_{k,i}^{\overline{l}_k}\right) + \nabla \phi_{i}\left(x_{k,i}^{\overline{l}_k}\right)^T\left(x_k - x_{k,i}^{\overline{l}_k}\right) \\
& - \left[\phi_{r+1}\left(x_{k,r+1}^{\overline{l}_k}\right) + \nabla \phi_{r+1}\left(x_{k,r+1}^{\overline{l}_k}\right)^T\left(x_k - x_{k,r+1}^{\overline{l}_k}\right)\right]\text,~i\in \{1,\hdots,r\}\text,\\
\end{split}
\end{equation*}   
and
\begin{equation*}
\tilde{J}_k := 
\left(\begin{array}{c}
\nabla \phi_{1}\left(x_{k,1}^{\overline{l}_k}\right)^T - \nabla \phi_{r+1}\left(x_{k,r+1}^{\overline{l}_k}\right)^T\\
\vdots \\
\nabla \phi_{{r}}\left(x_{k,r}^{\overline{l}_k}\right)^T - \nabla \phi_{r+1}\left(x_{k,r+1}^{\overline{l}_k}\right)^T
\end{array}\right)\text. 
\end{equation*}
So, the minimization problem~\eqref{eq:qp-max} can be viewed as a quadratic approximation of
\begin{equation}\label{eq:opt-phi}
\begin{split}
\min_{x\in \mathbb{R}^{n}~~} & \phi_{r+1}(x) \\
\text{s.t.~~} & \Phi(x) = 0\text,
\end{split} 
\end{equation} 
where
\begin{equation*}
\Phi(x) := 
\left(\begin{array}{c}
\phi_{1}(x) - \phi_{r+1}(x)\\
\vdots \\
\phi_{r}(x) - \phi_{r+1}(x)
\end{array}\right)\text. 
\end{equation*} 

With this initial analysis, we are ready to understand why we have chosen to design a method that produces a sequence of function values that is not monotonically decreasing. When one tries to move superlinearly to  a solution of a smooth constrained optimization problem, the Maratos effect~\cite{MAR78,MIM05} must be taken into consideration. Sometimes, a good movement towards $x_*$ might be not accepted because the candidate for the next iterate does not improve the function value. Normally, a correction step is made to prevent this undesirable property to happen and the superlinear convergence can be assured. 

We have seen above that the quadratic problem that is solved in Step 2 can be seen as a smooth constrained optimization problem and one might expect that we can do the same correction step to ensure a superlinear movement towards the solution. However, since we suppose that we do not know the functions $\phi_i$, such a correction becomes very hard to perform. One could try to numerically approximate $\tilde{J}_k$ during the execution of the algorithm to create a correction step, but this estimation can be very tricky.   For these reasons, we have chosen, for some specific iterations, to accept the step computed by our method without giving attention to the function value. As we will see later, this choice allows us to maintain a superlinear convergence result.    
 
Notice that for any $s\in \set U(x)$, we have that $f$ behaves smoothly along $s$ at $x$, since the $s$-directional derivatives of $\phi_i$ are all the same for $i\in \set I(x)$. Consequently, the kernel of the Jacobian of $\Phi(x)$ will be of great importance to us, because it tends to recover the smooth subspace of $f$ at $x_*$ when $x$ approaches $x_*$. Therefore, we denote by $J_x$ the Jacobian of $\Phi(x)$ and by $Z^\triangleleft_x$ the matrix whose columns form a basis for the kernel of $J_x$. Moreover, from now on, our analysis will be restricted to the case that $r\in \{1,\hdots,n-1\}$. The cases $r = 0$ and $r = n$ will be treated later (see Remark~\ref{remark:3}).

In light of Remark~\ref{remark:1}, due to Assumption~\ref{assump:affine}, it is possible to see that the map $J_x:\mathbb{R}^n\rightarrow \mathbb{R}^r$ is surjective for all $x$ in a small neighborhood $\set N$ of $x_*$. Hence, for $x\in \set N$, there must exist $J^\triangleleft_x\in \mathbb{R}^{n\times r}$ such that $J_xJ^\triangleleft_x = I_{r}$. Moreover, by~\cite[Lemma 14.3]{BGL06}, one can see that there is only one map 
\begin{equation*}
\begin{split}
Z:\mathbb{R}^n & \longrightarrow \mathbb{R}^{(n-r)\times n}\\
x & \longmapsto Z_{x}
\end{split}
\end{equation*}
such that $Z_xJ^\triangleleft_x$ is a null matrix, $Z_xZ^\triangleleft_x = I_{n-r}$ and the following relations hold
\begin{equation}\label{eq:2subspaces}
Z^\triangleleft_xZ_x + J^\triangleleft_xJ_x = I_n\text{ and }J_xZ^\triangleleft_x = 0\text. 
\end{equation}
So, we may divide $\mathbb{R}^n$ into two subspaces, generated by the columns of $Z^\triangleleft_x$ and $J^\triangleleft_x$, respectively.

Now, coming back to the optimization problem~\eqref{eq:opt-phi}, we define its Lagrangian function $\mathcal L(x,\lambda):\mathbb{R}^n\times\mathbb{R}^r\rightarrow \mathbb{R}$ as 
\begin{equation}\label{eq:lagrangian}
\mathcal L(x,\lambda) = \phi_{r+1}(x) + \lambda^T\Phi(x)\text.
\end{equation} 
By Remark~\ref{remark:1}, the feasible set of problem~\eqref{eq:opt-phi} satisfies the linear independence constraint qualification and thus there is only one $\lambda_*\in \mathbb{R}^r$ such that $\nabla_x \mathcal L(x_*,\lambda_*)$ is the null vector. So, in possession of this vector $\lambda_*$, we define $g:\mathbb{R}^n\rightarrow \mathbb{R}^{n-r}$, where
\begin{equation}\label{eq:dev-reduced}
g(x) := {Z^\triangleleft_x}^T\nabla_x \mathcal L(x,\lambda_*) \overset{\eqref{eq:2subspaces}}{=} {Z^\triangleleft_x}^T\nabla\phi_{r+1}(x)\text.
\end{equation}
Moreover, for not overloading the proofs that will follow, we also define
\begin{equation} \label{eq:Amatrix}
A_{k}:= I_n - Z^\triangleleft_{x_k}\hat{H}_k^{-1}{Z^\triangleleft_{x_k}}^TH_k\text,
\end{equation}
with
\begin{equation*}
\hat{H}_k:= {Z^\triangleleft_{x_k}}^TH_kZ^\triangleleft_{x_k}\text.
\end{equation*}

Below, we present a theorem that establishes the exact solution $d_{k,\overline{l}_k}$ obtained in~\eqref{eq:qp-max} whenever it is equivalent to~\eqref{eq:equivalent}. For this result and the subsequent ones, we define 
\begin{equation}\label{eq:define-tau}
\tau_{k,\overline{l}_k} := \max_{1\leq i\leq r+1}\left\| x_{k,i}^{\overline{l}_k} - x_k\right\|\text.
\end{equation}
\begin{theorem}\label{theo:direction}
Under Assumptions~\ref{assump1}, \ref{assump:affine} and~\ref{assump:strongmin}, suppose we are at a fixed outer iteration $k$ of GraFuS and at the last inner iteration indexed by $\overline{l}_k$. Then, if $k\in \set K$, where $\set K$ is the index set defined in~\eqref{eq:K-set}, and $x_k\in \set N$, where $\set N$ is the small neighborhood in which the map $J_x$ is surjective, we have that
\begin{equation*}
d_{k,\overline{l}_k} = d_{k,\overline{l}_k}^{\set U} + d_{k,\overline{l}_k}^{\set V}\text,
\end{equation*}
where
\begin{equation*}
d_{k,\overline{l}_k}^{\set U} := -Z^\triangleleft_{x_k}\hat{H}_k^{-1}g(x_k) + \rho_{k}^{\set U}~~\text{and}~~d_{k,\overline{l}_k}^{\set V} := -A_{k}J^\triangleleft_{x_k}\Phi(x_k) + \rho_{k}^{\set V}\text,
\end{equation*}
with
\begin{equation*}
\rho_k^{\set U} = -Z^\triangleleft_{x_k}\hat{H}_k^{-1}{Z^\triangleleft_{x_k}}^T\overline\rho_k~~\text{and}~~\rho_k^{\set V} = -A_{k}J^\triangleleft_{x_k}\hat\rho_k\text, 
\end{equation*}
for some $\overline\rho_k\in \mathbb{R}^n$ and $\hat\rho_k\in \mathbb{R}^r$ satisfying
\begin{equation*}
\|\overline\rho_k\| = O\left(\tau_{k,\overline{l}_k}\right)~~\text{and}~~\|\hat\rho_k\| = O\left(\tau_{k,\overline{l}_k}^{2}\right) + O\left(\tau_{k,\overline{l}_k}\right)O\left(\|d_{k,\overline{l}_k}\|\right)\text,
\end{equation*}
with $\tau_{k,\overline{l}_k}$ defined in~\eqref{eq:define-tau}.
\end{theorem} 
\begin{proof}
First, we consider the Karush-Kuhn-Tucker conditions of problem~\eqref{eq:equivalent}, which tell us that the solution $d_{k,\overline{l}_k}$ must satisfy
\begin{equation}\label{eq:KKT1}
\tilde{\Phi}_k + \tilde{J}_kd_{k,\overline{l}_k} = 0
\end{equation}
and 
\begin{equation}\label{eq:KKT2}
\nabla \phi_{{r+1}}\left(x_{k,r+1}^{\overline{l}_k}\right) + H_kd_{k,\overline{l}_k} + \tilde{J}_k^T\tilde\lambda = 0\text,
\end{equation}
for some $\tilde\lambda \in \mathbb{R}^r$. Since the functions that comprise $f$ satisfy $\phi_i\in C^2$, for $i\in \{1,\hdots,p\}$, we have, by relations~\eqref{eq:KKT1}~and~\eqref{eq:KKT2} that 
\begin{equation}\label{eq:KKT1-aprox}
\begin{split}
0 & = \Phi(x_k) + J_{x_k}d_{k,\overline{l}_k} + \left[\tilde{\Phi}_k - \Phi(x_k)\right] + \left[\tilde{J}_k - J_{x_k}\right]d_{k,\overline{l}_k}\\
& = \Phi(x_k) + J_{x_k}d_{k,\overline{l}_k} + \hat\rho_k
\end{split}
\end{equation}
and 
\begin{equation}\label{eq:KKT2-aprox}
\nabla \phi_{{r+1}}\left(x_k\right) + H_kd_{k,\overline{l}_k} + J_{x_k}^T\tilde\lambda + \overline\rho_k = 0\text,
\end{equation} 
where $\|\hat\rho_k\| = O\left(\tau_{k,\overline{l}_k}^{2}\right) + O\left(\tau_{k,\overline{l}_k}\right)O\left(\|d_{k,\overline{l}_k}\|\right)$ and $\|\overline\rho_k\| =  O\left(\tau_{k,\overline{l}_k}\right)$.
Then, because $A_{k}J^\triangleleft_{x_k}$ is a right inverse for $J_{x_k}$ (see \cite[Section 14.2]{BGL06} or simply use the fact that $J_{x_k}Z_{x_k}^\triangleleft = 0$), it is possible to decompose $\mathbb{R}^n$ in two subspaces generated by the columns of $Z^\triangleleft_{x_k}$ and $A_{k}J^\triangleleft_{x_k}$. As a consequence, we can consider two vectors $d_{k,\overline{l}_k}^{\set U}$ and $d_{k,\overline{l}_k}^{\set V}$ such that there exist $\alpha_{\set U}$ and $\alpha_{\set V}$ that imply
\begin{equation*}
d_{k,\overline{l}_k} = d_{k,\overline{l}_k}^{\set U} + d_{k,\overline{l}_k}^{\set V}\text,
\end{equation*}
with
\begin{equation*}
d_{k,\overline{l}_k}^{\set U} = Z^\triangleleft_{x_k}\alpha_{\set U}~~\text{and}~~d_{k,\overline{l}_k}^{\set V} = A_{k}J^\triangleleft_{x_k}\alpha_{\set V}\text.
\end{equation*}
Hence, looking at relation~\eqref{eq:KKT1-aprox}, we obtain that
\begin{equation*}
\alpha_{\set V} = -\Phi(x_k) - \hat\rho_k\text, 
\end{equation*}
which yields
\begin{equation*}
d_{k,\overline{l}_k}^{\set V} = -A_{k}J^\triangleleft_{x_k}\Phi(x_k) + \rho_k^{\set V}\text{,~~with~~} \rho_k^{\set V} = -A_{k}J^\triangleleft_{x_k}\hat\rho_k\text.
\end{equation*} 
Finally, pre-multiplying relation~\eqref{eq:KKT2-aprox} by ${Z_{x_k}^\triangleleft}^T$, we have
\begin{equation*}
g(x_k) + {Z^\triangleleft_{x_k}}^TH_k\left[Z^\triangleleft_{x_k}\alpha_U - A_{k}J^\triangleleft_{x_k}\left(\Phi(x_k) + \hat\rho_k\right)\right] + {Z^\triangleleft_{x_k}}^T\overline\rho_k = 0\text.
\end{equation*}
Then, since ${Z^\triangleleft_{x_k}}^TH_kA_{k} = 0$, we complete the proof by noticing that
\begin{equation*}
\alpha_{\set U} = -\hat{H}_k^{-1}g(x_k) -\hat{H}_k^{-1}{Z^\triangleleft_{x_k}}^T\overline\rho_k \Rightarrow d_{k,\overline{l}_k}^{\set U} = -Z^\triangleleft_{x_k}\hat{H}_k^{-1}g(x_k) + \rho_k^{\set U}\text,
\end{equation*}
where $\rho_k^{\set U} = -Z^\triangleleft_{x_k}\hat{H}_k^{-1}{Z^\triangleleft_{x_k}}^T\overline\rho_k$. 
\end{proof}

%With this theorem in hand, we are able to prove a simple corollary.
%\begin{corollary}\label{cor:homo}
%Under the assumptions of Theorem~\ref{theo:direction}, we have that 
%\begin{equation*}
%\|\Phi(x_{k+1})\| = O(\nu_k^2)\text.
%\end{equation*}
%\end{corollary}
%\begin{proof}
%Since $\phi_i\in C^2$, $\|d_{k,\overline{l}_k}\|\leq \nu_k$ and $\tau_{k,\overline{l}_k} \leq \nu_k $, it yields that
%\begin{equation*}
%\begin{split}
%\|\Phi(x_{k+1})\| & \leq \|\Phi(x_k) + J_{x_k}d_{k,\overline{l}_k}\| + O(\nu_k^2) \\
%                  & \leq \|\Phi(x_k) - \Phi(x_k)\| + \|\hat\rho_k\| + O(\nu_k^2)\\
%                  &   =  O(\nu_k^2)\text, 
%\end{split}
%\end{equation*}
%which is the desired result. \hfill$\Box$
%\end{proof}
%
%The previous statement leaves us with an important observation: for iterations of GraFuS that satisfy $k\in \set K$, the homogeneous system $\Phi(x) = 0$ is quickly satisfied, since $\nu_k$ is associated with our optimality certificate (notice that the term $O(\nu_k^2)$ in Corollary~\ref{cor:homo} could also be changed to $o(\epsilon_{k,\overline{l}_k})$ or $o(\Delta_{k,\overline{l}_k})$ without losing validity).  

Below, we present the last technical result before providing the key theorem of this subsection. As an hypothesis of this statement, we assume that the matrices $H_k$ must converge to a matrix $H_*$, where 
\begin{equation}\label{eq:H-convergence}
H_* =  \nabla^2_{xx}\mathcal{L}(x_*,\lambda_*) + \gamma J_{x_*}^TJ_{x_*}\text,~~\text{for some }\gamma\geq0\text. 
\end{equation}
By Assumption~\ref{assump:strongmin}, we see that the Hessian of the Lagrangian must be a positive definite matrix with respect to the subspace $\set U(x_*)$. So, for $\gamma > 0$ sufficiently large, $H_*$ becomes also a positive definite matrix.

\begin{theorem}\label{theo:superlinear}
Under Assumptions~\ref{assump1}, \ref{assump:affine} and~\ref{assump:strongmin}, suppose that $x_k\rightarrow x_*$, where $x_*\in \mathbb{R}^n$ is a local minimizer for $f$ presented in~\eqref{eq:minimax-prob}. Assume that $k\in \set K$, where $\set K$ is the index set defined in~\eqref{eq:K-set}, and $x_k\in \set N$, where $\set N$ is the small neighborhood in which the map $J_x$ is surjective. Also, close to $x_*$, suppose that the maps
\begin{equation*}
\begin{array}{cccc}
Z^\triangleleft: & \mathbb{R}^n & \longrightarrow & \mathbb{R}^{n\times(n-r)} \\
                 &      x       & \longmapsto     &  Z^\triangleleft_{x}
\end{array}
\text{~~~and~~~}
\begin{array}{cccc}
J^\triangleleft: & \mathbb{R}^n & \longrightarrow & \mathbb{R}^{n\times r} \\
                 &      x       & \longmapsto     &  J^\triangleleft_{x}
\end{array}
\end{equation*}
are all Lipschitz continuous functions and that the reduced gradient given in~\eqref{eq:dev-reduced} satisfies $g\in C^1$ with $g'$ being also a Lipschitz continuous function close to $x_*$. Moreover, assume that $H_k\rightarrow H_*$ with $H_*$ being the matrix presented in~\eqref{eq:H-convergence}. Additionally, suppose that, close to $x_*$, we have $\|H_k - H_*\| = O(\|x_k - x_*\|)$. Then, the following relation holds
\begin{equation*}
\|x_{k+1} - x_*\| = O(\|x_k-x_*\|^2) + \rho_k^{\set U} + \rho_k^{\set V}\text,~~\text{for }k\in \set K\text, 
\end{equation*} 
with $\rho_k^{\set U}$ and $\rho_k^{\set V}$ from Theorem~\ref{theo:direction}.
\end{theorem}
\begin{proof}
First, let us define $\tilde{x}_{k+1} := x_k + d_{k,\overline{l}_k}^{\set V}$, with $k\in \set K$. Now, observe that, from the definition~\eqref{eq:Amatrix}, for $x_k$ close enough to $x_*$, we have $\|A_{k} - A_{*}\| = O(\|x_k - x_*\|)$, where 
\begin{equation*}
A_{*}:= I_n - Z^\triangleleft_{x_*}\hat{H}_*^{-1}{Z^\triangleleft_{x_*}}^TH_*\text{,~~with~}\hat{H}_* := {Z^\triangleleft_{x_*}}^TH_*Z^\triangleleft_{x_*}\text.
\end{equation*}
Using this fact, considering the Taylor expansion of the map $\Phi$ around $x_*$ and remembering that $\Phi(x_*) = 0$ in the equality $(*)$ below and noticing that $J^\triangleleft$ is Lipschitz continuous and a bounded map around $x_*$ in $(**)$, we have, for a sufficiently small neighborhood of $x_*$, that 
\begin{equation*}
\begin{split}
\tilde{x}_{k+1} - x_* & = x_k - x_* -A_{k}J^\triangleleft_{x_k}\Phi(x_k) + \rho_k^{\set V} \\
& \overset{(*)}{=} x_k - x_* - A_{k}J^\triangleleft_{x_k}J_{x_*}(x_k-x_*) + O(\|x_k-x_*\|^2) + \rho_k^{\set V} \\
& = x_k - x_* - A_{*}J^\triangleleft_{x_*}J_{x_*}(x_k-x_*) + O(\|x_k-x_*\|^2) + \rho_k^{\set V} \\
& ~~~- \left[A_{k}\left(J^\triangleleft_{x_k}-J^\triangleleft_{x_*}\right) + \left(A_{k} - A_{*}\right)J^\triangleleft_{x_*}\right]J_{x_*}(x_k-x_*) \\
& \overset{(**)}{=} x_k - x_* -A_{*}J^\triangleleft_{x_*}J_{x_*}(x_k-x_*) + O(\|x_k-x_*\|^2) + \rho_k^{\set V} \text. 
\end{split}
\end{equation*}
Consequently, taking into account the relation (see \cite[Section 14.5]{BGL06})
\begin{equation*}
g'(x_*) = {Z^\triangleleft_{x_*}}^T\nabla^2_{xx}\mathcal{L}(x_*,\lambda_*)
\end{equation*}
in $(\bullet)$, the Lipschitz property around $x_*$ of the maps $Z^\triangleleft$ and $\hat{H}^{-1}$ in $(\bullet\bullet)$, the relation~\eqref{eq:H-convergence} in $(\blacktriangle)$ and the relation~\eqref{eq:2subspaces} in $(\blacktriangle\blacktriangle)$, we have
\begin{equation*}
\begin{split}
x_{k+1} - x_* & = x_{k} + d_{k,\overline{l}_k}^{\set V} + d_{k,\overline{l}_k}^{\set U} - x_* \\
& = \tilde{x}_{k+1} - x_* -Z^\triangleleft_{x_k}\hat{H}_k^{-1}g(x_k)+\rho_k^{\set U}~\text{(recall that }\tilde{x}_{k+1} := x_k + d_{k,\overline{l}_k}^{\set V}\text)\\
& \overset{(\bullet)}{=} \tilde{x}_{k+1} - x_* -Z^\triangleleft_{x_k}\hat{H}_k^{-1}{Z^\triangleleft_{x_*}}^T\nabla^2_{xx}\mathcal{L}(x_*,\lambda_*)(x_k-x_*) \\
&~~~+ O(\|x_k-x_*\|^2)+\rho_k^{\set U} \\
& \overset{(\bullet\bullet)}{=} \tilde{x}_{k+1} - x_* -Z^\triangleleft_{x_*}\hat{H}_*^{-1}{Z^\triangleleft_{x_*}}^T\nabla^2_{xx}\mathcal{L}(x_*,\lambda_*)(x_k-x_*) \\
&~~~+ O(\|x_k-x_*\|^2)+\rho_k^{\set U} \\
& \overset{(\blacktriangle)}{=} \tilde{x}_{k+1} - x_* -Z^\triangleleft_{x_*}\hat{H}_*^{-1}{Z^\triangleleft_{x_*}}^TH_*(x_k-x_*) + O(\|x_k-x_*\|^2)+\rho_k^{\set U} \\
& = A_{*}(x_k-x_*) -A_{*}J^\triangleleft_{x_*}J_{x_*}(x_k-x_*) + O(\|x_k-x_*\|^2) \\
&~~~+ \rho_k^{\set U} + \rho_k^{\set V} \\ 
& = A_{*}(I-J^\triangleleft_{x_*}J_{x_*})(x_k-x_*) + O(\|x_k-x_*\|^2) + \rho_k^{\set U} + \rho_k^{\set V} \\
& \overset{(\blacktriangle\blacktriangle)}{=} A_{*}Z^\triangleleft_{x_*}Z_{x_*}(x_k-x_*) + O(\|x_k-x_*\|^2) + \rho_k^{\set U} + \rho_k^{\set V}\text.
\end{split}
\end{equation*}
Hence, since $A_{*}Z^\triangleleft_{x_*} = 0$, it yields that
\begin{equation*}
\|x_{k+1} - x_*\| = O(\|x_k-x_*\|^2) + \rho_k^{\set U} + \rho_k^{\set V}\text,
\end{equation*}
which concludes the proof. 
\end{proof}

Finally, we are able to prove the most important result of this manuscript, which ensures that, under special circumstances, the method either moves superlinearly to a minimizer of the problem or superlinearly reduces the optimality certificate.
\begin{theorem}\label{theo:key}
Under Assumptions~\ref{assump1}, \ref{assump:affine} and~\ref{assump:strongmin}, suppose that $\{x_k\}$ is an infinite sequence generated by GraFuS with $\nu_k\rightarrow 0$ and that we are under the conditions of Theorem~\ref{theo:superlinear}. Then, if $\sigma_k > 1$ for all $k\in \set K$, we have 
\begin{equation*}
\min\left\{ \frac{\nu_{k+1}}{\nu_k} , \frac{\|x_{k+1} - x_*\|}{\|x_{k} - x_*\|} \right\} \underset{k\in \set K}{\rightarrow} 0\text.
\end{equation*}
\end{theorem}
\begin{proof}
Suppose, by contradiction, that there exist an infinite index set $\hat{\set K}\subset \set K$ and $M > 0$ such that
\begin{equation}\label{eq:absurd-main}
\min\left\{ \frac{\nu_{k+1}}{\nu_k} , \frac{\|x_{k+1} - x_*\|}{\|x_{k} - x_*\|} \right\} > M\text{, for all  } k\in \hat{\set K}\text.
\end{equation}
Therefore,
\begin{equation*}
\frac{\nu_{k+1}}{\nu_k} > M\text{, for all  } k\in \hat{\set K}\text,
\end{equation*}
which yields, by the way the algorithm was designed, that 
\begin{equation*}
\|H_k^{-1}G_{k,\overline{l}_k}\lambda_{k,\overline{l}_k}\| \geq M\nu_k\text{, for all  } k\in \hat{\set K}\text.  
\end{equation*}
Now, since $\epsilon_k = O(\nu_k)$ and $d_{k,\overline{l}_k} = -H_k^{-1}G_{k,\overline{l}_k}\lambda_{k,\overline{l}_k}$ is a valid relation for the primal-dual variables that solve the quadratic programming problem that appears in Step~2 (when the trust-region constraints are not active), we have that
\begin{equation*}
\epsilon_k = O(\|d_{k,\overline{l}_k}\|) = O(\|x_{k+1} - x_{k}\|)\text{, for all  } k\in \hat{\set K}\text.
\end{equation*}
Therefore, because~\eqref{eq:absurd-main} is assumed, we have
\begin{equation*}
\frac{\|x_{k+1} - x_{k}\|}{\|x_{k+1} - x_*\|} \leq 1 + \frac{\|x_k - x_*\|}{\|x_{k+1} - x_*\|} < 1 + \frac{1}{M} = \frac{M + 1}{M}\text,
\end{equation*}
which assures that $\epsilon_k = O(\|x_{k+1} - x_*\|)$. Consequently, since $\sigma_k > 1$, it yields that
\begin{equation*}
\tau_{k,\overline{l}_k} = O\left(\left(\epsilon_{k,\overline l_k}\right)^{\sigma_k}\right) = o(\|x_{k+1} - x_*\|)\text,
\end{equation*} 
and, by Theorem~\ref{theo:direction}, we see that
\begin{equation*}
\|\hat\rho_k\| = o(\|x_{k+1} - x_*\|)~~\text{and}~~\|\overline\rho_k\| = o(\|x_{k+1} - x_*\|),
\end{equation*}
which ensures, by Theorem~\ref{theo:superlinear}, that
\begin{equation*}
\|x_{k+1} - x_*\| = O(\| x_k - x_* \|^2) + o\left(\|x_{k+1} - x_*\|\right)\text{, for all  } k\in \hat{\set K}\text.
\end{equation*} 
So, for all $k\in \hat{\set K}$ sufficiently large, the following holds
\begin{equation*}
\|x_{k+1} - x_*\| = O(\| x_k - x_* \|^2)\text.
\end{equation*}
However, the above relation contradicts the initial assumption~\eqref{eq:absurd-main}. Therefore, we must have
\begin{equation*}
\min\left\{ \frac{\nu_{k+1}}{\nu_k} , \frac{\|x_{k+1} - x_*\|}{\|x_{k} - x_*\|} \right\} \underset{k\in \set K}{\rightarrow} 0\text.
\end{equation*}
\end{proof}

%The above result is a strong statement and it should be properly understood. It guarantees that, when $\nu_{k+1}/\nu_k > M$, for some fixed $M>0$, the iterates $x_k$, with $k\in \set K$, go to $x_*$ superlinearly. Notice that $\nu_{k+1}/\nu_k \to 0$ only when the optimality certificate $\nu_k$ is outdated, i.e., although a better value for the optimality certificate is clearly possible, it remains stuck to the same value for many iterations (mainly because many iterations have passed without a good sampling). Roughly speaking, it says that if the method has been sampling well, GraFuS will move superlinearly to the solution of the problem in infinitely many iterations. We believe that this is the best result that one can expect for a method that obtains its information about the objective function by a sampling procedure.     

\begin{remark}\label{remark:3}
The local convergence results were developed assuming $r\in \{1,\hdots,n-1\}$. For the case $r = 0$, we have that the method is approaching a point for which the function $f$ is smooth in the whole neighborhood. For such a situation, it is straightforward to see that the direction $d_{k,\overline{l}_k}$ will have only the $\set U$-component, i.e., $d_{k,\overline{l}_k} = d_{k,\overline{l}_k}^{\set U}$ with $Z^\triangleleft_{x} = I_n$ for all $x$ around $x_*$. Now, considering $r = n$, we see that the method is approaching a point where $f$ is nonsmooth in any direction. For that case, it is also clear that the direction $d_{k,\overline{l}_k}$ will have only the $\set V$-component, i.e., $d_{k,\overline{l}_k} = d_{k,\overline{l}_k}^{\set V}$ with $A_k \equiv I_n$ for all $x_k$ around $x_*$. Therefore, in both cases, the result of Theorem~\ref{theo:key} will be preserved, but, for the case that $r = n$, the value $\sigma_k$ does not need to be strictly greater than one, i.e., in such a case Theorem~\ref{theo:key} holds for $\sigma_k = 1$.
\end{remark} 

\section{Numerical Results}

This section has the intent to illustrate the main local convergence results obtained. However, by no means we had the ambition to present an extensive set of tests nor to recommend our method over any other one. Here, our main goal is to provide the reader with proof-of-concept numerical results. 

All the problems were solved using Matlab in an Intel Core 2 Duo T6500, 2.10 GHz and 4 Gb of RAM. We have used \texttt{quadprog} as the tool for solving the quadratic minimizations needed in each iteration, setting \texttt{active-set} as the algorithmic choice and $10^{-12}$ as the tolerances \texttt{TolX} and \texttt{TolFun} and $10^{-8}$ (default value) as \texttt{TolCon}. Moreover, for all functions we have chosen random starting points such that $\|x_0\|_\infty \leq 2$ and solved each of them twenty times in order to have statistical relevance of the results.

We have solved each optimization problem with two algorithms: (i) the GS method presented by the original authors~\cite{BLO05} but with a nonnormalized search direction (a variant introduced by Kiwiel~\cite{KWL07}, that has the advantage to asymptotically recover the steepest descent method when applied to smooth functions) and (ii) the GraFuS method. We have used the original GS implementation without any modification (with the exception of using a nonnormalized search direction)\footnotemark\footnotetext{The GS code can be found at http://cs.nyu.edu/overton/papers/gradsamp/alg/.}. For completeness, we present the parameter values used in Algorithm~\ref{model-alg}: $m = 2n$; $\nu_0 = 10^{-6}$; $\epsilon_0 = 10^{-1}$; $\nu_\text{opt} = 10^{-6}$; $\epsilon_\text{opt} = 10^{-6}$; $\theta_\nu = 1$; $\theta_\epsilon = 10^{-1}$; $\gamma = 0.5$; $\beta = 0$ and $\alpha_k = 1$. 

%The implementation of GraFuS is also available
%\footnotemark\footnotetext{The authors freely provide the GraFuS code:\\ 
%\attachfile[icon=Tag]{GraFuS.m}{~~GraFuS code}
%\attachfile[icon=Tag]{toyprob.m}{~~Test function}
%\attachfile[icon=Tag]{start.m}{~~Script}} 
%and we have used it to produce the numerical results. 

The parameter values used in GraFuS were: $m = 2n$; $\nu_0 = 10^{-2}$; $\nu_\text{opt} = 10^{-6}$; $\gamma_\epsilon = 4$; $\gamma_\Delta = 4$; $\delta = 0.90$; $\varrho = 1.50$; $\rho = 10^{-8}$ and $\theta = 0.5$. The value of $\sigma_k$ in Step 1 was set as follows. We start the algorithm with $\sigma_0 = 1$ and, setting $|\lambda|_{\#}$ as the number of entries of $\lambda$ greater than $10^{-3}/(n+1)$, we have updated $\sigma_k$ every time a reduction on $\nu_k$ was performed in such a way that
\begin{equation*}
\sigma_{k+1} = \left\{
\begin{array}{cc}
1\text, & |\lambda_{k,\overline{l}_k}|_{\#} \geq n + 1\\
1.5\text, & \text{otherwise}
\end{array}\right.\text.
\end{equation*}
Notice that $|\lambda|_{\#} - 1$ tries to approximate the dimension of the subspace $\set V(x_*)$.    

An important aspect that we must recall here is that the iterations of GraFuS are more expensive than those of GS. While the GS routine finds a search direction and does an Armijo line search to find the next iterate, GraFuS constantly solves quadratic programming problems until it finds a good set of sampled points and a good trust region to move. Therefore, one could take advantage of the way GS was designed as a bootstrap to start performing GraFuS iterations, deciding if the current iterate is close to the solution indirectly by means of the size of the current sampling radius. As a result, we only start to run the GraFuS algorithm after the second reduction of the sampling radius in GS (i.e. when $\epsilon_k < 10^{-2}$), and that is the reason why in the figures that follow below, we see that in the first iterations both methods remain together. 

We also must stress that although the optimality certificates of Algorithms~\ref{model-alg}~and~\ref{vugs-alg} are very similar, they are not the same (specially because the quadratic programming problem of each method is different). Therefore, one might be more rigorous than the other one. Thus, although in most problems the GraFuS method appears to be closer to the solution, this does not mean that GS is not able to reach the same precision (maybe a tighter optimality parameter would allow it). 

Finally, the way we have chosen the matrices $H_k$ is a delicate matter and, for that reason, we have reserved the following subsection to explain our procedure. It is worth pointing out that we have used BFGS ideas to update the matrices, but we do not have any theoretical guarantee that the matrices $H_k$ will converge to a matrix of the form presented in~\eqref{eq:H-convergence}. Nevertheless, the choice on how we update the matrices has a strong foundation, since it uses the same reasoning of a Sequential Quadratic Programming (SQP) updating~\cite{GMS05} for the optimization problem that appears in~\eqref{eq:opt-phi}.  

\subsection{$H_k$ updates in GraFuS method}\label{subsec:Hk_update}

As we have seen in the last section, if some hypotheses are satisfied, it is possible to see the quadratic programming problem that is solved in every iteration of GraFuS as a smooth constrained optimization problem. Moreover, the matrix that we would like to approximate (at least in its null space) is the Hessian of~\eqref{eq:lagrangian}. Therefore, a natural attempt to reach that goal is to update the positive definite matrix $H_k$ as it is done in SQP routines. In other words, it would be desirable to have the following relation
\begin{equation*}
H_k(x_+ - x_-) = \nabla_{x} \mathcal{L}(x_+,\lambda_+) - \nabla_{x} \mathcal{L}(x_-,\lambda_-)\text,
\end{equation*}
where $\mathcal{L}$ is the Lagrangian function defined in~\eqref{eq:lagrangian} and $\lambda_+$ and $\lambda_-$ are vectors that try to approximate the multiplier $\lambda_*$ that fulfills~\eqref{eq:dev-reduced}. In addition,
\begin{equation*}
\begin{split}
\nabla_{x} \mathcal{L}(x,\lambda) & =  \nabla \phi_{r+1}(x) + \sum_{i=1}^r\lambda_i(\nabla \phi_i(x) - \nabla \phi_{r+1}(x)) \\
& = \left(1 - \sum_{i=1}^r\lambda_i\right)\nabla \phi_{r+1}(x) + \sum_{i=1}^r\lambda_i\nabla \phi_i(x)\text.
\end{split}
\end{equation*}
Therefore, defining $\hat\lambda\in \mathbb{R}^{r+1}$ as $\hat\lambda_i = \lambda_i$, for $i \in \{1,\hdots,r\}$, and
\begin{equation*} 
\hat\lambda_{r+1} = 1 - \sum_{i=1}^r\lambda_i\text,
\end{equation*}
we have $e^T\hat\lambda = 1$ and one can rewrite $\nabla_{x} \mathcal{L}(x,\lambda) = \hat G\hat \lambda$, where
\begin{equation*}
\hat G := [\nabla \phi_1(x)~\hdots \nabla \phi_{r+1}(x)]\text.
\end{equation*}
Hence, if in two fixed pairs $(k_+,l_+)$ and $(k_-,l_-)$ of (outer,~inner) iterations we have good sets of sampled points (in the sense that the conditions $i)$ and $ii)$ of Lemma~\ref{lemma:good-sample} are valid), it is natural to ask that the following secant relationship holds 
\begin{equation*}
H_k(x_{k_+} - x_{k_-}) = G_{k_+,l_+}\lambda_{k_+,l_+} - G_{k_-,l_-}\lambda_{k_-,l_-}\text.
\end{equation*}
The problem here is how one can identify a good set of sampled points. In fact, because of Lemma~\ref{lemma:good-sample}, we could say that all iterations in $\set K$ must produce a good set of sampled points, but to restrain the update of $H_k$ just for those iterations can lead us to very few updates during the execution of the method.  So, although there is no straightforward response, we know that a good set of sampled points is associated with a small norm of the convex combination of its gradients. Hence, a good strategy would be to update the matrix $H_k$ only if such a condition is verified.

Based on the previous reasoning, we present next the routine that provides the sequence of matrices $H_k$ that are used within GraFuS.
\vspace{0.5cm}
 
\begin{itemize}[leftmargin=1.5cm]
\item[Step 0.]{Start setting $H = I$ and let the GraFuS algorithm run until it finds two pairs $(k_+,l_+)$ and $(k_-,l_-)$ of (outer,~inner) iterations such that 
\begin{equation*}
\left\|G_{k_+,l_+}\lambda_{k_+,l_+}\right\|\leq \sqrt{\nu_{k_+}}\text{~~and~~}\left\|G_{k_-,l_-}\lambda_{k_-,l_-}\right\|\leq \sqrt{\nu_{k_-}}\text.
\end{equation*}
Set 
$$x_+ := x_{k_+} \text{~and~} x_- := x_{k_-}\text;$$ 
$$v_+ := G_{k_+,l_+}\lambda_{k_+,l_+} \text{~and~} v_- := G_{k_-,l_-}\lambda_{k_-,l_-}\text.$$
}
\item[Step 1.]{Set $p: = x_+ - x_-$ and $q := v_+ - v_-$. If $q^Tp < 0.2p^THp$ then compute a new vector $q$ by Powell's correction (see \cite[Subsection 18.2]{BGL06}).}
\item[Step 2.]{Update $H$:
\begin{equation*}
H \leftarrow H - \frac{Hpp^TH}{p^THp} + \frac{qq^T}{q^Tp}\text.
\end{equation*}
}
\item[Step 3.] Use the subsequent matrices $H_k$ as $H$ until the GraFuS algorithm finds another iteration $\hat k$ and an inner iteration $\hat l$ such that
\begin{equation*}
\left\|G_{\hat k,\hat l}\lambda_{\hat k,\hat l}\right\|\leq \sqrt{\nu_{\hat k}}\text.
\end{equation*}
Then, $x_- \leftarrow x_{+}$, $x_+ \leftarrow x_{\hat k}$, $v_- \leftarrow v_+$, $v_+ \leftarrow  G_{\hat k,\hat l}\lambda_{\hat k,\hat l}$. Go back to Step 1.  
\end{itemize} 
\vspace{0.5cm}

Clearly, other ways of updating $H_k$ are possible. Indeed, even the pure BFGS update as considered in~\cite{LEO13} can be performed (although, in such a case, we have to assume that for all iterates the function $f$ will be differentiable and Assumption~\ref{assump1} will no longer be satisfied). We believe that an improvement on the updating of $H_k$ may be an important advance on the performance of GraFuS. 

\subsection{Illustrative examples}

The functions that were solved to illustrate our algorithm are the following~\cite{HMM04}:
\begin{itemize}[leftmargin=1cm]
\item[\textbf{F1)}]{Chained CB3 I
\begin{equation*}
f(x) = \sum_{i=1}^{n-1} \max\left\{ x_i^4 + x_{i+1}^2, (2-x_i)^2 + (2-x_{i+1})^2, 2\exp(-x_i+x_{i+1}) \right\}\text;
\end{equation*}
}
\item[\textbf{F2)}]{Chained CB3 II 
\begin{equation*}
\begin{split}
f(x) = & \max\left\{ \sum_{i=1}^{n-1} \left( x_i^4 + x_{i+1}^2 \right) \right., \sum_{i=1}^{n-1} \left( (2-x_i)^2 + (2-x_{i+1})^2 \right),\\
& \left. \sum_{i=1}^{n-1} 2\exp(-x_i+x_{i+1}) \right\}\text;
\end{split}
\end{equation*}
}
\item[\textbf{F3)}]{Nonsmooth generalization of Brown function 2
\begin{equation*}
f(x) = \sum_{i = 1}^{n-1}\left( |x_i|^{x_{i+1}^2 + 1} + |x_{i+1}|^{x_i^2 + 1} \right);
\end{equation*}
}
\item[\textbf{F4)}]{Chained crescent I 
\begin{equation*}
\begin{split}
f(x) = & \max\left\{ \sum_{i=1}^{n-1} \left( x_i^2 + (x_{i+1} - 1)^2 + x_{i+1} - 1 \right) \right., \\ 
& \left. \sum_{i=1}^{n-1} \left( -x_i^2 - (x_{i+1} - 1)^2 + x_{i+1} + 1 \right) \right\}\text.
\end{split}
\end{equation*}
}
\end{itemize}

The first two functions are convex, whereas the last two ones are nonconvex functions. In addition, \textbf{F1} and \textbf{F3} satisfy $\set U(x_*) = \{0\}$, a condition that does not hold for \textbf{F2} and \textbf{F4}.

To observe the GraFuS functioning and to put it into perspective with GS, we have comparatively examined the CPU time and the number of iterations versus $f(x_k) - f_*$, where $f_*$ is set as the best function value obtained by the methods in all of the runs. Because of the nondeterministic nature of the methods, we have used the median and quartiles (25\% and 75\%) of the twenty runs. As a complementary tool for assessing how fast our method goes towards the optimal function value, in the plots with the number of iterations, we have represented the value
\begin{equation*}
\min\left\{\frac{f(x_{k+1}) - f_*}{f(x_k) - f_*},1\right\}
\end{equation*}
with color scales along the plotted curves of GraFuS, where a brighter hue stands for a value close to zero, and a darker color for the values near one. Notice that the values above $1$ must have a safeguard, because since our method might not be monotone, the ratio $(f(x_{k+1}) - f_*)/(f(x_k) - f_*)$ could be greater than one. The reader will see that a few increases on the function value appears in the figures that are shown below. This is due to the fact that the measure (quartiles) used to represent the twenty runs somehow absorbs the nonmonotone behavior of the function values.

Additionally, we have examined the values
\begin{equation*}
\frac{\nu_{k+1}}{\nu_k} ~~\text{and}~~ \frac{\|x_{k+1} - x_*\|}{\|x_{k} - x_*\|} 
\end{equation*}
for $k$ such that $\nu_{k+1} < \nu_k$. For this two measures, a detailed explanation must be given on how we have plotted the corresponding curves. As we have mentioned before, we have solved each function more than once. However, an iteration $k$ for which $\nu_{k+1} < \nu_k$ occurs is not necessarily the same iteration where a second run will have $\nu_{k+1} < \nu_k$. It is only possible to track these values for different runs if instead of looking at the iteration $k$, we monitor the actual occurrences of $\nu_{k+1} < \nu_k$. Therefore, we have proceeded in the following way. For each run, we set the $w$-dimensional vectors
\begin{equation*}
\text{vec}_\nu \leftarrow \left[\frac{\nu_{k_1+1}}{\nu_{k_1}},\hdots, \frac{\nu_{k_w+1}}{\nu_{k_w}}\right]~\text{and}~\text{vec}_{x_*} \leftarrow \left[\frac{\|x_{k_1+1} - x_*\|}{\|x_{k_1} - x_*\|},\hdots, \frac{\|x_{k_w+1} - x_*\|}{\|x_{k_w} - x_*\|}\right]\text,
\end{equation*}   
where $k_i$ is the iteration that for the $i$-th time, $\nu_{k_i+1} < \nu_{k_i}$ has occurred. Moreover, for the case that $w < 30$, we enlarge the vectors $\text{vec}_\nu$ and $\text{vec}_{x_*}$ by copying the last value of each vector, respectively, until it reaches 30 dimensions. This is necessary because not every run of GraFuS will give vectors with equal dimensions. Then, the quartiles are computed using the vectors $\text{vec}_\nu$ and $\text{vec}_{x_*}$ of each run. 
 
\begin{figure}[h!]
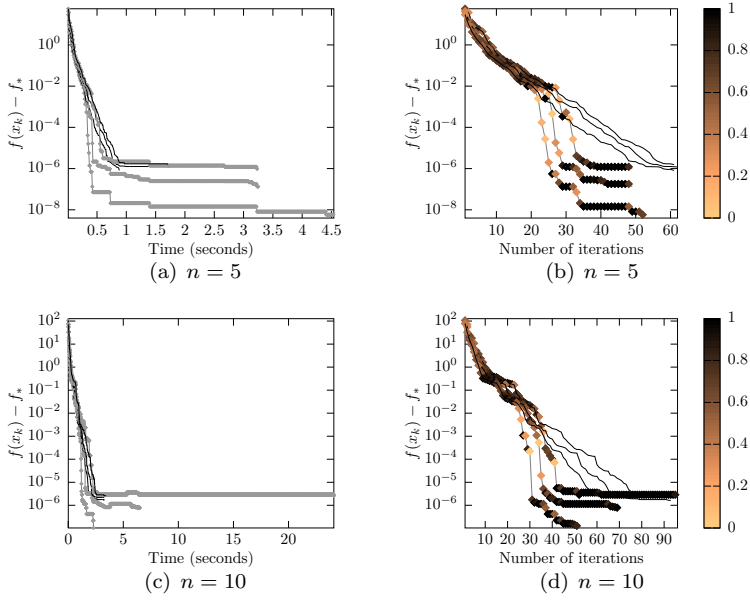

\centering
\subfigure[$n = 5$]{\resizebox{0.35\textwidth}{!}{\input{fig-chainedcb31-n5-time.tex}}}\hspace{0.6cm}
\subfigure[$n = 5$]{\resizebox{0.35\textwidth}{!}{\input{fig-chainedcb31-n5-it.tex}}}
\subfigure[$n = 10$]{\resizebox{0.35\textwidth}{!}{\input{fig-chainedcb31-n10-time.tex}}}\hspace{0.6cm}
\subfigure[$n = 10$]{\resizebox{0.35\textwidth}{!}{\input{fig-chainedcb31-n10-it.tex}}}
\caption{Medians and quartiles of twenty runs of GS and GraFuS methods for the function \textbf{F1}. The black line plots represent the GS method, whereas the grey/colored continuous ones with $\lozenge$ marks stand for GraFuS. For both $n = 5$ and $n = 10$, we have $x_* = e$.}
\label{fig:chainedcb31}
\end{figure}

\begin{figure}[h!]
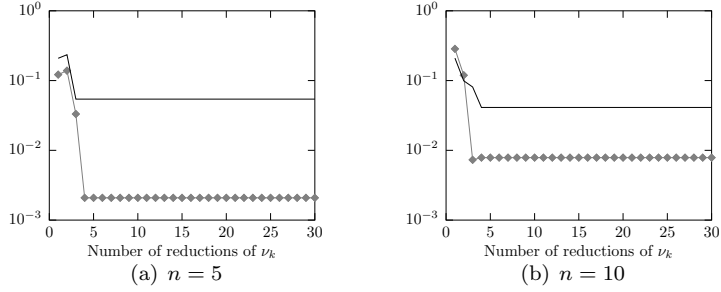

\centering
\subfigure[$n = 5$]{\resizebox{0.35\textwidth}{!}{\input{fig-chainedcb31-n5-dist.tex}}}\hspace{0.6cm}
\subfigure[$n = 10$]{\resizebox{0.35\textwidth}{!}{\input{fig-chainedcb31-n10-dist.tex}}}
\caption{The simple black line plot and the one with $\lozenge$ marks represent, respectively, the medians of the vectors $\text{vec}_\nu$ and $\text{vec}_{x_*}$ for the function \textbf{F1}. For both $n = 5$ and $n = 10$, we have $x_* = e$.}
\label{fig:chainedcb31-dist}
\end{figure}

In Figures~\ref{fig:chainedcb31} and~\ref{fig:chainedcb31-dist}, we see the results obtained by the runs related to the first function \textbf{F1}. It is possible to observe that GraFuS has a good performance in all the measures. Not only a high precision is achieved, but one can also see 
\begin{equation}\label{eq:seq-zero}
\min\left\{ \frac{\nu_{k+1}}{\nu_k} , \frac{\|x_{k+1} - x_*\|}{\|x_{k} - x_*\|} \right\}
\end{equation}
approaching zero, in accordance with the result of Theorem~\ref{theo:key}. On the other hand, although Figure~\ref{fig:chainedcb32} presents good results for the function \textbf{F2}, the sequence that appears in Theorem~\ref{theo:key} does not approach zero as fast as it happens for \textbf{F1} (see Figure~\ref{fig:chainedcb32-dist}). This has a reasonable explanation. Notice that the result of Theorem~\ref{theo:key} is conditioned by Theorem~\ref{theo:superlinear}, which has, as an assumption, that the matrices $H_k$ must converge to $H_*$ satisfying~\eqref{eq:H-convergence}. However, $H_*$ must converge to the Hessian of the Lagrangian only with respect to the subspace $\set U(x_*)$. Since \textbf{F1} has $\set U(x_*) = \{0\}$, the matrices $H_k$ do not need to contain any kind of second-order information to guarantee Theorem~\ref{theo:key} to hold, which is the reason why~\eqref{eq:seq-zero} approaches quickly to zero when GraFuS is applied to this function. However, \textbf{F2} has $\dim \set U(x_*) = n - 2$, which means that the result of Theorem~\ref{theo:key} will be conditioned to how good is the approximation of $H_k$ to $H_*$ at each iteration. 

\begin{figure}[h!]
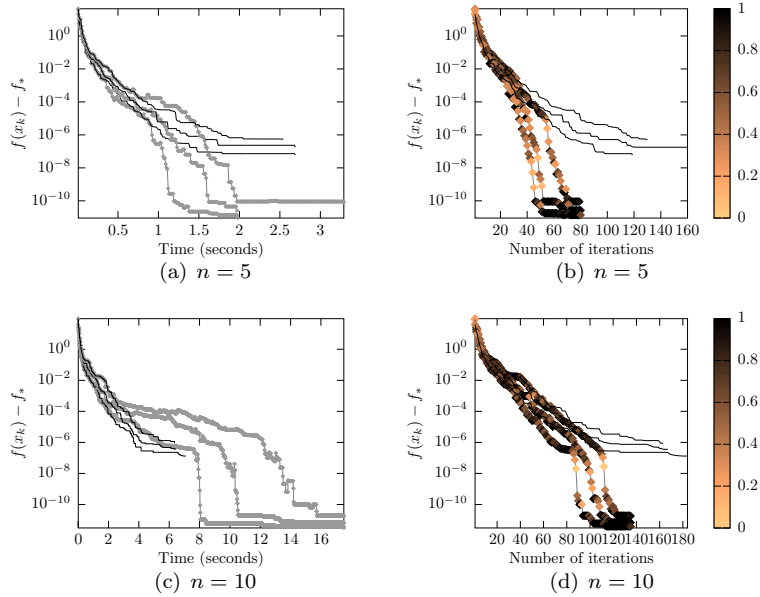

\centering
\subfigure[$n = 5$]{\resizebox{0.35\textwidth}{!}{\input{fig-chainedcb32-n5-time.tex}}}\hspace{0.6cm}
\subfigure[$n = 5$]{\resizebox{0.35\textwidth}{!}{\input{fig-chainedcb32-n5-it.tex}}}
\subfigure[$n = 10$]{\resizebox{0.35\textwidth}{!}{\input{fig-chainedcb32-n10-time.tex}}}\hspace{0.6cm}
\subfigure[$n = 10$]{\resizebox{0.35\textwidth}{!}{\input{fig-chainedcb32-n10-it.tex}}}
\caption{Medians and quartiles of twenty runs of GS and GraFuS methods for the function \textbf{F2}. The black line plots represent the GS method, whereas the grey/colored continuous ones with $\lozenge$ marks stand for GraFuS. For both $n = 5$ and $n = 10$, we have $x_* = e$.}
\label{fig:chainedcb32}
\end{figure}

\begin{figure}[h!]
\centering
\subfigure[$n = 5$]{\resizebox{0.35\textwidth}{!}{\input{fig-chainedcb32-n5-dist.tex}}}\hspace{0.6cm}
\subfigure[$n = 10$]{\resizebox{0.35\textwidth}{!}{\input{fig-chainedcb32-n10-dist.tex}}}
\caption{The simple black line plot and the one with $\lozenge$ marks represent, respectively, the medians of the vectors $\text{vec}_\nu$ and $\text{vec}_{x_*}$ for the function \textbf{F2}. For both $n = 5$ and $n = 10$, we have $x_* = e$.}
\label{fig:chainedcb32-dist}
\end{figure}

\begin{figure}[h!]
\centering
\subfigure[$n = 5$]{\resizebox{0.35\textwidth}{!}{\input{fig-brownfunction2-n5-time.tex}}}\hspace{0.6cm}
\subfigure[$n = 5$]{\resizebox{0.35\textwidth}{!}{\input{fig-brownfunction2-n5-it.tex}}}
\subfigure[$n = 10$]{\resizebox{0.35\textwidth}{!}{\input{fig-brownfunction2-n10-time.tex}}}\hspace{0.6cm}
\subfigure[$n = 10$]{\resizebox{0.35\textwidth}{!}{\input{fig-brownfunction2-n10-it.tex}}}
\caption{Medians and quartiles of twenty runs of GS and GraFuS methods for the function \textbf{F3}. The black line plots represent the GS method, whereas the grey/colored continuous ones with $\lozenge$ marks stand for GraFuS. For both $n = 5$ and $n = 10$, we have $x_* = 0$.}
\label{fig:brownfunction2}
\end{figure}

\begin{figure}[h!]
\centering
\subfigure[$n = 5$]{\resizebox{0.35\textwidth}{!}{\input{fig-brownfunction2-n5-dist.tex}}}\hspace{0.6cm}
\subfigure[$n = 10$]{\resizebox{0.35\textwidth}{!}{\input{fig-brownfunction2-n10-dist.tex}}}
\caption{The simple black line plot and the one with $\lozenge$ marks represent, respectively, the medians of the vectors $\text{vec}_\nu$ and $\text{vec}_{x_*}$ for the function \textbf{F3}. For both $n = 5$ and $n = 10$, we have $x_* = 0$.}
\label{fig:brownfunction2-dist}
\end{figure}

\begin{figure}[h!]
\centering
\subfigure[$n = 5$]{\resizebox{0.35\textwidth}{!}{\input{fig-chainedcrescent1-n5-time.tex}}}\hspace{0.6cm}
\subfigure[$n = 5$]{\resizebox{0.35\textwidth}{!}{\input{fig-chainedcrescent1-n5-it.tex}}}
\subfigure[$n = 10$]{\resizebox{0.35\textwidth}{!}{\input{fig-chainedcrescent1-n10-time.tex}}}\hspace{0.6cm}
\subfigure[$n = 10$]{\resizebox{0.35\textwidth}{!}{\input{fig-chainedcrescent1-n10-it.tex}}}
\caption{Medians and quartiles of twenty runs of GS and GraFuS methods for the function \textbf{F4}. The black line plots represent the GS method, whereas the grey/colored continuous ones with $\lozenge$ marks stand for GraFuS. For both $n = 5$ and $n = 10$, we have $x_* = 0$.}
\label{fig:chainedcrescent1}
\end{figure}

\begin{figure}[h!]
\centering
\subfigure[$n = 5$]{\resizebox{0.35\textwidth}{!}{\input{fig-chainedcrescent1-n5-dist.tex}}}\hspace{0.6cm}
\subfigure[$n = 10$]{\resizebox{0.35\textwidth}{!}{\input{fig-chainedcrescent1-n10-dist.tex}}}
\caption{The simple black line plot and the one with $\lozenge$ marks represent, respectively, the medians of the vectors $\text{vec}_\nu$ and $\text{vec}_{x_*}$ for the function \textbf{F4}. For both $n = 5$ and $n = 10$, we have $x_* = 0$.}
\label{fig:chainedcrescent1-dist}
\end{figure}

The function \textbf{F3} has some interesting features, since it does not admit a maximum representation. Indeed, let us consider the function $h(a,b) = a^{(1+b^2)}$, for $a\geq 0$. Then, it yields that
\begin{equation*}
\lim_{\varepsilon\downarrow 0}\frac{\partial h}{\partial a}(\varepsilon,\varepsilon) = \lim_{\varepsilon\downarrow 0}(1+\varepsilon^2)\varepsilon^{\varepsilon^2} = 1\text;
\end{equation*}
\begin{equation*}
\lim_{\varepsilon\downarrow 0}\frac{\partial h}{\partial a}(2^{-1/\varepsilon^3},\varepsilon) = \lim_{\varepsilon\downarrow 0} (1+\varepsilon^2)2^{-1/\varepsilon} = 0\text.
\end{equation*}
So, it is possible to see that any representation of \textbf{F3} that might involve a maximum of functions cannot have smooth functions. Therefore, this function does not satisfy the requirements of our convergence analysis. However, this does not prevent GraFuS to have a good performance (see Figures~\ref{fig:brownfunction2} and~\ref{fig:brownfunction2-dist}).

Finally, looking at the results obtained for the function \textbf{F4} in Figures~\ref{fig:chainedcrescent1} and~\ref{fig:chainedcrescent1-dist}, the analysis follows very closely the one that was presented for function \textbf{F2}. Since \textbf{F4} satisfies $\set U(x_*) = n - 1$, as depicted in Figure~\ref{fig:chainedcrescent1-dist}, the value~\eqref{eq:seq-zero} does not go to zero as quickly as it could be expected. Nevertheless, the results obtained in Figure~\ref{fig:chainedcrescent1} also show a good behavior of GraFuS.

\section{Conclusions} 

This manuscript presents an implementable algorithm for solving unconstrained nonsmooth and nonconvex optimization problems. Using the ideas of the Gradient Sampling algorithm and taking advantage of some notions developed over the years for the Bundle Method, we were able to produce an algorithm that, in some sense, can be viewed as a generalization of the well established Newton's (quasi-Newton) method for nonconvex nonsmooth unconstrained minimization.

Additionally, we believe that an important step has been taken in the direction of obtaining a rapid method for minimizing nonconvex and nonsmooth functions. It was shown that a rapid move towards the solution is a reliable behavior for some iterations of GraFuS. Moreover, at least for the illustrative examples considered in the numerical experiments, one can see that fast moves are not rare and can be expected for a reasonable amount of iterations. However, it must be stressed that the iterations of GraFuS are computationally expensive when compared to GS, and for this reason, the rapid behavior of GraFuS might not be translated to a faster method for some functions.

The matters of efficiency and applicability of the method have not been treated properly in this manuscript, since our aim here was, first, to produce a mathematical theory that would support a rapid convergence to a solution and second, to provide proof-of-concept numerical instances that corroborate the main theoretical results. There are many possibilities of improvements on the algorithm (e.g. different forms of updating the matrices $H_k$ and efficient ways of selecting the sampled points without affecting the global convergence) and we hope that future studies explore these possibilities.

Finally, we end these final remarks with two questions that naturally arise from some of the numerical results obtained in the previous section: 
\begin{itemize}
\item{under which conditions could we establish $\|H_k-H_*\| = O(\|x_k - x_*\|)$ in Theorem~\ref{theo:superlinear}?}
\item{would it be possible to have convergence results with more general assumptions?} 
\end{itemize}

%\appendix

\section*{Appendix}

The aim of this appendix is to show that the assumption $(\lambda_{k,\overline{l}_k})_i = 0$, whenever $i \notin \set I(x_*)$ and $k\in \set K$, with $\set K$ defined in~\eqref{eq:K-set}, is not necessary. For this goal, we will show that even without such an assumption, the results from the local convergence subsection remain the same.  

We divide our reasoning in two cases and remind the reader that we have assumed $\set I(x_*) = \{1,\hdots,r+1\}$:
\begin{itemize}[leftmargin=1cm]
\item[A1)]{The cardinality of $\set I(x_*)$ is $n+1$;}
\item[A2)]{The cardinality of $\set I(x_*)$ is $r+1$ with $r < n$.}
\end{itemize}

Suppose first that A1 holds and let us consider an iterate $x_k$ sufficiently close to $x_*$. Moreover, assume that $k\in \set K$, where $\set K$ is the index set defined in~\eqref{eq:K-set}. Then, looking at the optimization problem in~\eqref{eq:equivalent}, we see that any additional active constraint will generate an additional active constraint to~\eqref{eq:equivalent} in a way that it will be a linear combination of the first $n+1$ active constraints (by Remark~\ref{remark:1} and because the rank of $\tilde{J}_{k}$ remains constant in a close neighborhood of $x_*$). Hence, the solution obtained with, or without, this additional constraint is the same, which yields that the results presented at the local convergence subsection do not change for this special case.  

So, let us consider the more intricate case A2. Moreover, let us assume that there is only one additional constraint, i.e., the number of active constraints is $r+2$ (we will see that the occurrence of more than one additional constraint will be a straightforward generalization of this simpler case). In other words, we are saying that solving~\eqref{eq:qp-max} is equivalent to minimize
\begin{equation*}
\begin{split}
\min_{\left(d,z\right)\in \mathbb{R}^{n+1}~~} & z + \frac{1}{2}d^TH_kd \\
\text{s.t.~~} & f\left(x_{k,i}^{\overline{l}_k}\right) + \nabla f\left(x_{k,i}^{\overline{l}_k}\right)^T\left(x_k + d - x_{k,i}^{\overline{l}_k}\right) = z_\text,~~1\leq i\leq r+2\text,
\end{split} 
\end{equation*}
where here we assume that rearrangements were done in order to have the additional constraint as the $(r+2)$-th constraint and that it has the associated sampled point $x_{k,r+2}^{\overline{l}_k}$. Therefore, for an iterate $x_k$ sufficiently close to the solution and a sufficiently small sampling radius, we have, by the continuity of the functions $\phi_i$, that only the functions $\phi_1,\hdots,\phi_{r+1}$ can assume the maximum at any sampled point. So, there exists $j\in \{1,\hdots,r+1\}$ such that $f(x_{k,r+2}^{\overline{l}_k}) = \phi_j(x_{k,r+2}^{\overline{l}_k})$. Consequently, recalling that $k\in \set K$, the above minimization problem can be seen as
\begin{equation*}
\begin{split}
\min_{\left(d,z\right)\in \mathbb{R}^{n+1}~~} & z + \frac{1}{2}d^TH_kd \\
\text{s.t.~~} & \phi_i\left(x_{k,i}^{\overline{l}_k}\right) + \nabla \phi_i\left(x_{k,i}^{\overline{l}_k}\right)^T\left(x_k + d - x_{k,i}^{\overline{l}_k}\right) = z_\text,~~1\leq i\leq r+1\\
& \phi_j\left(x_{k,r+2}^{\overline{l}_k}\right) + \nabla \phi_j\left(x_{k,r+2}^{\overline{l}_k}\right)^T\left(x_k + d - x_{k,r+2}^{\overline{l}_k}\right) = z_\text,
\end{split} 
\end{equation*}
whose dual optimization problem is written as
\begin{equation}\label{eq:uncon}
\begin{split}
\max_{\lambda\in \mathbb{R}^{r+2}~~} & \sum_{i=1}^{r+1} \lambda_i\left[ \phi_i\left(x_{k,i}^{\overline{l}_k}\right) + \nabla \phi_i\left(x_{k,i}^{\overline{l}_k}\right)^T\left(x_k - x_{k,i}^{\overline{l}_k}\right) \right] \\ 
&+ \lambda_{r+2}\left[ \phi_j\left(x_{k,r+2}^{\overline{l}_k}\right) + \nabla \phi_j\left(x_{k,r+2}^{\overline{l}_k}\right)^T\left(x_k - x_{k,r+2}^{\overline{l}_k}\right) \right] \\ & - \frac{1}{2}\left\| \sum_{i=1}^{r+1} \lambda_i \nabla \phi_i(x_{k,i}^{\overline{l}_k}) + \lambda_{r+2} \nabla \phi_j(x_{k,r+2}^{\overline{l}_k}) \right\|_{H_k^{-1}}^2\\
\text{s.t.~~} & e^T\lambda = 1\text.
\end{split} 
\end{equation}
Therefore, we can turn this last constrained maximization problem into an unconstrained one by making the following substitution $\lambda_{r+2} = 1 - \sum_{i=1}^{r+1} \lambda_i$. So, we have
\begin{equation*}
\begin{split}
\max_{\lambda\in \mathbb{R}^{r+1}~~} & \sum_{i=1}^{r+1} \lambda_i\left[ \phi_i\left(x_{k,i}^{\overline{l}_k}\right) + \nabla \phi_i\left(x_{k,i}^{\overline{l}_k}\right)^T\left(x_k - x_{k,i}^{\overline{l}_k}\right) - \phi_j\left(x_{k,r+2}^{\overline{l}_k}\right) \right.\\
&\left. - \nabla \phi_j\left(x_{k,r+2}^{\overline{l}_k}\right)^T\left(x_k - x_{k,r+2}^{\overline{l}_k}\right) \right] + \phi_j\left(x_{k,r+2}^{\overline{l}_k}\right) \\ 
& + \nabla \phi_j\left(x_{k,r+2}^{\overline{l}_k}\right)^T\left(x_k - x_{k,r+2}^{\overline{l}_k}\right) \\
& - \frac{1}{2}\left\| \sum_{i=1}^{r+1} \lambda_i \left[\nabla \phi_i(x_{k,i}^{\overline{l}_k}) - \nabla \phi_j(x_{k,r+2}^{\overline{l}_k}) \right] + \nabla \phi_j(x_{k,r+2}^{\overline{l}_k}) \right\|_{H_k^{-1}}^2\text.\\
\end{split} 
\end{equation*}

Since the above problem is concave and smooth, its solution $\overline{\lambda}\in \mathbb{R}^{r+1}$ can be obtained by equalling the derivative of the objective function to the null vector. Consequently, assuming without loss of generality that the function $\phi_j$ involved in the additional constraint is $\phi_{r+1}$, we have
\begin{tiny}
\begin{equation*}
\begin{split}
&\left(
\begin{array}{c}
\nabla \phi_1\left( x_{k,1}^{\overline{l}_k} \right)^T - \nabla \phi_{r+1}\left(x_{k,r+2}^{\overline{l}_k}\right)^T\\
\vdots\\
\nabla \phi_{r+1}\left( x_{k,r+1}^{\overline{l}_k} \right)^T - \nabla \phi_{r+1}\left(x_{k,r+2}^{\overline{l}_k}\right)^T
\end{array}\right)H_k^{-1}
\left(
\begin{array}{c}
\nabla \phi_1\left( x_{k,1}^{\overline{l}_k} \right)^T - \nabla \phi_{r+1}\left(x_{k,r+2}^{\overline{l}_k}\right)^T\\
\vdots\\
\nabla \phi_{r+1}\left( x_{k,r+1}^{\overline{l}_k} \right)^T - \nabla \phi_{r+1}\left(x_{k,r+2}^{\overline{l}_k}\right)^T
\end{array}\right)^T\overline{\lambda} = \\
& \left(\begin{array}{c}
\phi_1\left(x_{k,1}^{\overline{l}_k}\right) + \nabla \phi_1\left(x_{k,1}^{\overline{l}_k}\right)^T\left(x_k - x_{k,1}^{\overline{l}_k}\right) \\
\vdots\\
\phi_{r+1}\left(x_{k,r+1}^{\overline{l}_k}\right) + \nabla \phi_{r+1}\left(x_{k,r+1}^{\overline{l}_k}\right)^T\left(x_k - x_{k,r+1}^{\overline{l}_k}\right)
\end{array}\right) \\
& - \left(\begin{array}{c}
\phi_{r+1}\left(x_{k,r+2}^{\overline{l}_k}\right) + \nabla \phi_{r+1}\left(x_{k,r+2}^{\overline{l}_k}\right)^T\left(x_k - x_{k,r+2}^{\overline{l}_k}\right)\\
\vdots\\
\phi_{r+1}\left(x_{k,r+2}^{\overline{l}_k}\right) + \nabla \phi_{r+1}\left(x_{k,r+2}^{\overline{l}_k}\right)^T\left(x_k - x_{k,r+2}^{\overline{l}_k}\right)
\end{array}\right) \\
&-\left(
\begin{array}{c}
\nabla \phi_1\left( x_{k,1}^{\overline{l}_k} \right)^T - \nabla \phi_{r+1}\left(x_{k,r+2}^{\overline{l}_k}\right)^T\\
\vdots\\
\nabla \phi_{r+1}\left( x_{k,r+1}^{\overline{l}_k} \right)^T - \nabla \phi_{r+1}\left(x_{k,r+2}^{\overline{l}_k}\right)^T
\end{array}\right)H_k^{-1}\nabla \phi_{r+1}\left(x_{k,r+2}^{\overline{l}_k}\right) \text. 
\end{split}
\end{equation*}
\end{tiny}
Now, changing the points $x_{k,r+2}^{\overline{l}_k}$ for $x_{k,r+1}^{\overline{l}_k}$ and redefining 
\begin{equation*}
\tau_{k,\overline{l}_k} := \max_{1\leq i\leq r+2}\left\| x_{k,i}^{\overline{l}_k} - x_k\right\|\text,
\end{equation*}
we get
\begin{tiny}
\begin{equation*}
\begin{split}
&\left(
\begin{array}{c}
\nabla \phi_1\left( x_{k,1}^{\overline{l}_k} \right)^T - \nabla \phi_{r+1}\left(x_{k,r+1}^{\overline{l}_k}\right)^T\\
\vdots\\
\nabla \phi_{r}\left( x_{k,r}^{\overline{l}_k} \right)^T - \nabla \phi_{r+1}\left(x_{k,r+1}^{\overline{l}_k}\right)^T\\
0^T
\end{array}\right)H_k^{-1}
\left(
\begin{array}{c}
\nabla \phi_1\left( x_{k,1}^{\overline{l}_k} \right)^T - \nabla \phi_{r+1}\left(x_{k,r+1}^{\overline{l}_k}\right)^T\\
\vdots\\
\nabla \phi_{r}\left( x_{k,r}^{\overline{l}_k} \right)^T - \nabla \phi_{r+1}\left(x_{k,r+1}^{\overline{l}_k}\right)^T\\
0^T
\end{array}\right)^T\overline{\lambda} = \\
& \left(\begin{array}{c}
\phi_1\left(x_{k,1}^{\overline{l}_k}\right) + \nabla \phi_1\left(x_{k,1}^{\overline{l}_k}\right)^T\left(x_k - x_{k,1}^{\overline{l}_k}\right) - \phi_{r+1}\left(x_{k,r+1}^{\overline{l}_k}\right) - \nabla \phi_{r+1}\left(x_{k,r+1}^{\overline{l}_k}\right)^T\left(x_k - x_{k,r+1}^{\overline{l}_k}\right)\\
\vdots\\
\phi_{r}\left(x_{k,r}^{\overline{l}_k}\right) + \nabla \phi_{r}\left(x_{k,r}^{\overline{l}_k}\right)^T\left(x_k - x_{k,r}^{\overline{l}_k}\right) - \phi_{r+1}\left(x_{k,r+1}^{\overline{l}_k}\right) - \nabla \phi_{r+1}\left(x_{k,r+1}^{\overline{l}_k}\right)^T\left(x_k - x_{k,r+1}^{\overline{l}_k}\right)\\
0^T
\end{array}\right)\\
&-\left(
\begin{array}{c}
\nabla \phi_1\left( x_{k,1}^{\overline{l}_k} \right)^T - \nabla \phi_{r+1}\left(x_{k,r+1}^{\overline{l}_k}\right)^T\\
\vdots\\
\nabla \phi_{r}\left( x_{k,r}^{\overline{l}_k} \right)^T - \nabla \phi_{r+1}\left(x_{k,r+1}^{\overline{l}_k}\right)^T\\
0^T
\end{array}\right)H_k^{-1}\nabla \phi_{r+1}\left(x_{k,r+1}^{\overline{l}_k}\right) + O\left(\tau_{k,\overline{l}_k}\right) \text. 
\end{split}
\end{equation*}
\end{tiny}
This last linear system yields
\begin{tiny}
\begin{equation*}
\begin{split}
&\left(
\begin{array}{c}
\nabla \phi_1\left( x_{k,1}^{\overline{l}_k} \right)^T - \nabla \phi_{r+1}\left(x_{k,r+1}^{\overline{l}_k}\right)^T\\
\vdots\\
\nabla \phi_{r}\left( x_{k,r}^{\overline{l}_k} \right)^T - \nabla \phi_{r+1}\left(x_{k,r+1}^{\overline{l}_k}\right)^T
\end{array}\right)H_k^{-1}
\left(
\begin{array}{c}
\nabla \phi_1\left( x_{k,1}^{\overline{l}_k} \right)^T - \nabla \phi_{r+1}\left(x_{k,r+1}^{\overline{l}_k}\right)^T\\
\vdots\\
\nabla \phi_{r}\left( x_{k,r}^{\overline{l}_k} \right)^T - \nabla \phi_{r+1}\left(x_{k,r+1}^{\overline{l}_k}\right)^T
\end{array}\right)^T
\left(\begin{array}{c}
\overline\lambda_1\\
\vdots\\
\overline\lambda_r
\end{array}\right) = \\
& \left(\begin{array}{c}
\phi_1\left(x_{k,1}^{\overline{l}_k}\right) + \nabla \phi_1\left(x_{k,1}^{\overline{l}_k}\right)^T\left(x_k - x_{k,1}^{\overline{l}_k}\right) - \phi_{r+1}\left(x_{k,r+1}^{\overline{l}_k}\right) - \nabla \phi_{r+1}\left(x_{k,r+1}^{\overline{l}_k}\right)^T\left(x_k - x_{k,r+1}^{\overline{l}_k}\right)\\
\vdots\\
\phi_{r}\left(x_{k,r}^{\overline{l}_k}\right) + \nabla \phi_{r}\left(x_{k,r}^{\overline{l}_k}\right)^T\left(x_k - x_{k,r}^{\overline{l}_k}\right) - \phi_{r+1}\left(x_{k,r+1}^{\overline{l}_k}\right) - \nabla \phi_{r+1}\left(x_{k,r+1}^{\overline{l}_k}\right)^T\left(x_k - x_{k,r+1}^{\overline{l}_k}\right)
\end{array}\right)\\
&-\left(
\begin{array}{c}
\nabla \phi_1\left( x_{k,1}^{\overline{l}_k} \right)^T - \nabla \phi_{r+1}\left(x_{k,r+1}^{\overline{l}_k}\right)^T\\
\vdots\\
\nabla \phi_{r}\left( x_{k,r}^{\overline{l}_k} \right)^T - \nabla \phi_{r+1}\left(x_{k,r+1}^{\overline{l}_k}\right)^T
\end{array}\right)H_k^{-1}\nabla \phi_{r+1}\left(x_{k,r+1}^{\overline{l}_k}\right) + O\left(\tau_{k,\overline{l}_k}\right) \text. 
\end{split}
\end{equation*}
\end{tiny}

Therefore, following the same reasoning used by us to get here, it is possible to see that the first $r$ components of the dual variable $\hat\lambda \in \mathbb{R}^{r+1}$ linked to the problem~\eqref{eq:just-active} must satisfy the last linear system obtained above (not considering the remaining error vector) and, moreover,
\begin{equation}\label{eq:c-lambda}
\hat\lambda_{r+1} = 1 - \sum_{i=1}^r \hat\lambda_i\text.
\end{equation} 
Therefore, considering $\lambda^*\in \mathbb{R}^{r+2}$ the solution of~\eqref{eq:uncon} and using equation~\eqref{eq:c-lambda}, we must have 
\begin{equation*}
\lambda^* = \left(\begin{array}{c}
\hat\lambda_1\\
\vdots\\
\hat\lambda_r\\
\lambda^*_{r+1}\\
1-\sum_{i=1}^r \hat\lambda_i - \lambda^*_{r+1}
\end{array}\right) + O\left(\tau_{k,\overline{l}_k}\right) = 
\left(\begin{array}{c}
\hat\lambda_1\\
\vdots\\
\hat\lambda_r\\
\lambda^*_{r+1}\\
\hat\lambda_{r+1} - \lambda^*_{r+1}
\end{array}\right) + O\left(\tau_{k,\overline{l}_k}\right)
\text.
\end{equation*}
So, to complete our reasoning, we write the following relation between the primal-dual variables
\begin{equation*}
\begin{split}
d_{k,\overline{l}_k} & = -H_k^{-1}\left[\sum_{i=1}^{r+1}\lambda_i^*\nabla \phi_i(x_{k,i}^{\overline{l}_k}) + \lambda^*_{r+2}\nabla \phi_{r+1}(x_{k,r+2}^{\overline{l}_k}) \right] \\
& = -H_k^{-1}\left[\sum_{i=1}^{r}\lambda^*_i\nabla \phi_i(x_{k,i}^{\overline{l}_k}) + \left(\lambda^*_{r+1} + \lambda^*_{r+2}\right)\nabla \phi_{r+1}(x_{k,r+1}^{\overline{l}_k}) \right] + O\left(\tau_{k,\overline{l}_k}\right) \\
& = -H_k^{-1}\sum_{i=1}^{r+1}\hat\lambda_i\nabla \phi_i(x_{k,i}^{\overline{l}_k}) + O\left(\tau_{k,\overline{l}_k}\right)\text. 
\end{split}
\end{equation*} 
Hence, $d_{k,\overline{l}_k}$ is exactly the search direction obtained in~\eqref{eq:just-active} with an additional error vector. Therefore, the term $O\left(\tau_{k,\overline{l}_k}\right)$ is absorbed by the other error vectors in Theorem~\ref{theo:superlinear} and the result is still valid.

Finally, remember that we have considered just one additional active constraint to the others $r+1$ active constraints. However, it is straightforward to see that exactly the same reasoning can be used to prove the result for any other number of additional constraints.
%\lipsum[71]

%\section*{Acknowledgments}
%We would like to acknowledge the assistance of volunteers in putting
%together this example manuscript and supplement.

\bibliographystyle{siamplain}
\bibliography{ref}
\end{document}